\documentclass[orivec]{llncs}
\usepackage{amsmath}
\usepackage{amsfonts}
\usepackage{amssymb}
\usepackage{graphicx}
\usepackage{stmaryrd}
\usepackage{color}
\usepackage{MnSymbol}

\newcommand{\cn}{\mathrm{Cn}}
\newcommand{\conp}[1]{\mathrm{Cn}(\{ {#1}\})}
\newcommand{\cons}[1]{\mathrm{Cn}(#1)}
\newcommand{\mb}[1]{\mathbb{#1}}

\newcommand{\mr}[1]{\mathrm{#1}}

\newtheorem{DEF} {Definition}
\newtheorem{OBS} {Observation}
\newtheorem{THE} {Theorem}
\newtheorem{LEM} {Lemma}

\newcommand{\relamod}{(\mathbb{X},\leqq)}

\newcommand{\mini}[1]{\mathbb{X}^{#1}_{<}}

\newcommand{\set}[1]{\mathbb{X}^{#1} }

\newcommand{\Belsome}[1]{\mathfrak{B}^{\vee}{#1}}

\newcommand{\preceqp}{\preceq_{p}}
\newcommand{\precc}{\prec_{c}}
\newcommand{\preceqc}{\preceq_{c}}
\newcommand{\simeqc}{\simeq_{c}}

\newcommand{\cro}{\ast_c}

\newcommand{\rtoo}{\mbox{$\langle \preceq$ $\mr{to}$ $\ast \rangle$}}

\newcommand{\mtod}{\mbox{$\langle \leqq$ $\mr{to}$ $\circ \rangle$}}

\newcommand{\dtoc}{\mbox{$\langle \circ$ $\mr{to}$ $\cro \rangle$}}

\newcommand{\mbtoc}{\mbox{$\langle \preceqc$ $\mr{to}$ $\cro \rangle$}}

\newcommand{\abtob}{\mbox{$\langle \preceq_{\ast}$ $\mr{to}$ $\preceq \rangle$}}
\newcommand{\mbtob}{\mbox{$\langle \preceqc$ $\mr{to}$ $\preceq \rangle$}}
\newcommand{\btomb}{\mbox{$\langle \preceq$ $\mr{to}$ $\preceqc \rangle$}}

\newcommand{\ctomb}{\mbox{$\langle \cro$ $\mr{to}$ $\preceqc \rangle$}}

\newcommand{\taut}{{\scriptstyle \top}}    
\newcommand{\falsum}{{\scriptstyle \perp}}

\begin{document}

\title{Choice revision}
\author{Li Zhang}
\date{}
\institute{}

\maketitle 
\thispagestyle{plain}



\begin{abstract}

\footnotesize
Choice revision is a sort of non-prioritized multiple revision, in which the agent partially accepts the new information represented by a set of sentences. We investigate the construction of choice revision based on a new approach to belief change called descriptor revision. We prove that each of two variants of choice revision based on such construction is axiomatically characterized with a set of plausible postulates, assuming that the object language is finite. Furthermore, we introduce an alternative modelling for choice revision, which is based on a type of relation on sets of sentences, named multiple believability relation. We show without assuming a finite language that choice revision constructed from such relations is axiomatically characterized with the same sets of postulates proposed for the choice revision based on descriptor revision, whenever the relations satisfy certain rationality conditions.
\\
\textbf{Keywords}: choice revision, non-prioritized multiple belief revision, belief change, descriptor revision, multiple believability relation 
\end{abstract}

\section{Introduction}

Belief change\footnote{\, In some literature, the term ``belief revision" is used as a synonym for belief change. In what follows, we use belief revision to refer to a particular kind of belief change.} theory studies how a rational agent changes her belief state when she is exposed to new information. Studies in this field have traditionally had a strong focus on two types of change: \emph{contraction} in which a specified sentence has to be removed from the original belief state, and \emph{revision} in which a specified sentence has instead to be consistently added. This paper is mainly concerned with the latter.

Alchourr\'on, G\"ardenfors and Makinson (AGM) performed the pioneering formal study on these two types of change in their seminal paper \cite{alchourron_logic_1985}. In the AGM theory of belief change, the agent's belief state is represented by  a set of sentences from some formal language $\mathcal{L}$, usually denoted by $K$. The new information is represented by a single sentence in $\mathcal{L}$. Belief revision and contraction on $K$ are formally represented by two operations $\ast$ and $\div$, mapping from a sentence $\varphi$ to a new set of sentences $K \ast \varphi$ and $K \div \varphi$ respectively. \cite{alchourron_logic_1985} postulated some conditions that a rational revision or contraction operation should satisfy, which are called AGM postulates on revision and contraction. 

Furthermore, \cite{alchourron_logic_1985} showed that contraction and revision satisfying AGM postulates could be precisely constructed from a model based on partial meet functions on remainder sets. After that, many alternative models \cite[etc.]{alchourron_logic_1985_safe,grove_two_1988,gardenfors_revisions_1988,hansson_kernel_1994} have been proposed to construct the contraction and revision operations satisfying AGM postulates. Although these models look entirely different on the surface, most of them employ the same select-and-intersect strategy \cite[p. 19]{hansson_descriptor_2017}. For example, in partial meet construction for contraction \cite{alchourron_logic_1985}, a selection is made among remainders and in sphere modelling for revision \cite{grove_two_1988}, a selection is made among possible worlds. The intersection of the selected objects is taken as the outcome of the operation in both cases.

Although the AGM theory has virtually become a standard model of theory change, many researchers are unsatisfied with its settings in several aspects and have proposed several modifications and generalizations to that framework (see \cite{ferme_agm_2011} for a survey). Here we only point out two inadequatenesses of the AGM theory.

On the one hand, in the original AGM model, the input is represented by a single sentence. This is unrealistic since agents often receive more than one piece of information at the same time. In order to cover these cases, we can generalize sentential revision to multiple revision, where the input is a finite or infinite set of sentences. On the other hand, in AGM revision, new information has priority over original beliefs. This is represented by the success postulate: $\varphi \in K \ast \varphi$ for all $\varphi$.  The priority means that the new information will always be entirely incorporated, whereas previous beliefs will be discarded whenever the agent need do so in order to incorporate the new information consistently. This is a limitation of AGM theory since in real life it is a common phenomenon that agents do not accept the new information that they receive or only accept it partially. As a modification, we can drop the success postulate and generalize prioritized revision to non-prioritized belief revision. 

In this contribution, we will put these two generalizations together and consider the so called multiple non-prioritized belief revision. In \cite{falappa_prioritized_2012}, two different kinds of such generalized revision are distinguished:

\begin{enumerate}
\item Merge: $K$ and $A$ are symmetrically treated, i.e., sentences of $K$ and $A$ could be accepted or rejected. 
\item Choice revision\footnote{\, Here we use the term ``choice revision'', introduced by \cite{Fuhrmann_phd}, to replace the term``selective change'' used in \cite{falappa_prioritized_2012}, for it is easier for us to distinguish it from the ``selective revision'' introduced in \cite{ferme_selective_1999}, which is a sort of non-prioritized revision with a single sentence as input. It should be noted that generally choice revision by a finite set $A$ cannot be reduced to selective revision by the conjunction $\& A$ of all elements in $A$. To see this, let $\ast_{s}$ be some selective revision. It is assumed that $\ast_{s}$ satisfies extensionality, i.e. if $\varphi  $ is logically equivalent to $ \psi$, then $K \ast_s \varphi = K \ast_s \psi$. So, $K \ast^\prime \& \{\varphi, \neg \varphi\} = K \ast^\prime \& \{\psi, \neg \psi\}$ for all $\varphi$ and $\psi$. However, it is clearly unreasonable for choice revision $\cro$ that $K \cro \{\varphi, \neg \varphi\} = K \cro \{\psi, \neg \psi\}$ should hold for all $\varphi$ and $\psi$.}: some sentences of A could be accepted, some others could be rejected.
\end{enumerate}

We use $\cro$ to denote a choice revision operation. \cite{falappa_prioritized_2012} investigated the formal properties of merge but left the study on choice revision as future work. As far as we know, little work has been done on this kind of revision in the literature. This fact can be partly explained by that the operation $\cro$ has the unusual characteristic that the standard select-and-intersect approach is not in general applicable. To see why, let the set $K$ of original beliefs not contain any element of $A =\{ \varphi, \neg \varphi \}$. We are going to construct a set $K \cro A$ which incorporates $\varphi$ or its negation. Suppose that we do that by first selecting a collection $\mb{X} =\{X_1, X_2, X_3, \cdots \}$ of sets of beliefs, each of which satisfies the success condition for choice revision with $A$, i.e. $X_i \cap A \neq \emptyset$ for each $X_i$. Then it may be the case that $\varphi \in X_1$ and $\neg \varphi \in X_2$. Given that $X_1$ and $X_2$ are consistent, it follows that the intersection $\cap \mb{X}$ cannot satisfy the success condition, i.e. it contains neither $\varphi$ or $\neg \varphi$.

Therefore, to develop a modelling for choice revision, we need to choose another strategy than the select-and-intersect method. \cite{hansson_descriptor_2013} introduced a new approach of belief change named ``descriptor revision'', which employs a ``select-direct'' methodology: It assumes that there is a set of belief sets as potential outcomes of belief change, and the belief change is performed by a direct choice among these potential outcomes. Furthermore, this is a very powerful framework for constructing belief change operations since success conditions for various types of belief changes are described in a general way with the help of a metalinguistic belief operator $\mathfrak{B}$. For instance, the success condition of contraction by $\varphi$ is $\neg \mathfrak{B} \varphi$, that of revision by $\varphi$ is $\mathfrak{B} \varphi$. Descriptor revision on a belief set $K$ is performed with a unified operator $\circ$ which applies to any success condition that is expressible with $\mathfrak{B}$. Hence, choice revision $\cro$ with a finite input set can be constructed from descriptor revision in the way of $K \cro \{\varphi_1, \varphi_2, \cdots, \varphi_n\} = K \circ \{\mathfrak{B} \varphi_1 \vee \mathfrak{B} \varphi_2 \vee \cdots \vee \mathfrak{B} \varphi_n \}$ \cite[p. 130]{hansson_descriptor_2017}. 

Although the construction of choice revision in the framework of descriptor revision has been introduced in \cite{hansson_descriptor_2017}, the formal properties of this type of belief change are still in need of investigation. The main purpose of this contribution is to conduct such an investigation. After providing some formal preliminaries in Section \ref{Preliminaries}, we will review how to construct choice revision in the framework of descriptor revision in Section \ref{section choice revison based on descriptor revision}. More importantly, in this section, we will investigate the postulates on choice revision which could axiomatically characterize these constructions. In Section \ref{section on an alternative modelling for choice revision} we will propose an alternative modelling for choice revision, which is based on \textit{multiple believability relations}, a generalized version of \textit{believability relation} introduced in \cite{hansson_relations_2014} and further studied in \cite{zhang_believability_2017}. We will investigate the formal properties of this modelling and prove the associated representation theorems.  Section \ref{section conclusion} concludes and indicates some directions for future work. All proofs of the formal results are placed in the appendix.

\section{Preliminaries}\label{Preliminaries}

The object language $\mathcal{L}$ is defined inductively by a set $v$ of propositional variables $\{p_0 ,\,  p_1, \, \cdots, \, p_n, \, \cdots \}$ and the truth-functional operations $\neg, \wedge, \vee$ and $\rightarrow$ in the usual way. $\taut$ is a tautology and $\falsum$ a contradiction. $\mathcal{L}$ is called finite if $v$ is finite. Sentences in $\mathcal{L}$ will be denoted by lower-case Greek letters and sets of such sentences by upper-case Roman letters. 

$\cn$ is a consequence operation for $\mathcal{L}$ satisfying supraclassicality (if $\varphi$ can be derived from $A$ by classical truth-functional logic, then $\varphi \in \cons{A}$), compactness (if $\varphi \in \cons{A}$, then there exists some finite $B \subseteq A$ such that $\varphi \in \cons{B}$) and the deduction property ($\varphi \in \cn(A \cup \{\psi\})$ if and only if $\psi \rightarrow \varphi \in \cn(A)$). $X\vdash \varphi$ and $X \nvdash \varphi$ are alternative notations for $\varphi \in \cn(X)$ and $\varphi \notin \cn(X)$ respectively. $\{ \varphi \} \vdash \psi$ is rewritten as $\varphi \vdash \psi$ for simplicity. And $\varphi \dashv \Vdash \psi$ means $\varphi \vdash \psi$ and $\psi \vdash \varphi$. $A \equiv B$ holds iff for every $\varphi \in A$, there exists some $\psi \in B$ such that $\varphi \dashv \Vdash \psi$ and vice versa. 

For all finite $A$, let $\& A$ denote the conjunction of all elements in $A$. For any $A$ and $B$, $A \owedge B = \{\varphi \wedge \psi \mid \varphi \in A \mbox{ and } \psi \in B\}$. We write $\varphi \owedge \psi$ and $A_0 \owedge A_1 \owedge \cdots \owedge A_n$ instead of $\{ \varphi\} \owedge \{\psi\}$ and $( \cdots ( A_0 \owedge A_1) \owedge \cdots) \owedge A_n$ for simplicity.

The set of beliefs an agent holds is represented by a \emph{belief set}, which is a set $X \subseteq \mathcal{L}$ such that $X = \cn(X)$. $K$ is fixed to denote the original beliefs of the agent. We assume that $K$ is consistent unless stated otherwise.


\section{Choice revision based on descriptor revision}\label{section choice revison based on descriptor revision}

Before investigating the properties of choice revision constructed in the framework of descriptor revision, we first present some formal basics of this framework, which is mainly based on \cite{hansson_descriptor_2013}. 

\subsection{Basics of descriptor revision}

An atomic belief descriptor is a sentence $\mathfrak{B}\varphi$ with $\varphi\in \mathcal{L}$. Note that the symbol $\mathfrak{B}$ is not part of the object language $\mathcal{L}$. A \emph{molecular belief descriptor} is a truth-functional combination of atomic descriptors. A composite belief descriptor (henceforth: descriptor; denoted by upper-case Greek letters) is a set of molecular descriptors. 

$\mathfrak{B}\varphi$ is \emph{satisfied} by a belief set $X$, if and only if $\varphi \in X$. Conditions of satisfaction for molecular descriptors are defined inductively, hence, provided that $\varphi$ and $\psi$ stand for molecular descriptors, $X$ satisfies $\neg \varphi$ if and only if it does not satisfy $\varphi$, it satisfies $\varphi \wedge \psi$ if and only if it satisfies both $\varphi$ and $\psi$, etc. It satisfies a composite descriptor $\Phi$ if and only if it satisfies all its elements. $X \Vdash \Phi$ denotes that set $X$ satisfies descriptor $\Phi$.


Descriptor revision on a belief set $K$ is performed with a unified operator $\circ$ such that $K \circ \Phi$ is an operation with $\Phi$ as its success condition. \cite{hansson_descriptor_2013} introduces several constructions for descriptor revision operations, of which the relational model defined as follows has a canonical status.

\begin{DEF}[\cite{hansson_descriptor_2013}]\label{definition of relational model}


$(\mb{X}, \leqq)$ is a relational select-direct model (in short: relational model) with respect to $K$ if and only if it satisfies:\footnote{\, We will drop the phrase ``with respect to $K$'' if this does not affect the understanding, and write $\set{\varphi}$ and $\mini{\varphi}$ instead of $\set{\{\mathfrak{B} \varphi\}}$ and $\mini{\{\mathfrak{B} \varphi\}}$ for simplicity.}
\begin{enumerate}
\item[]$(\mb{X} 1)$ $\mb{X}$ is a set of belief sets.
\item[]$(\mb{X} 2)$ $K \in \mb{X}$.
\item[]$(\leqq1)$ $K \leqq X$ for every $X \in \mb{X}$.
\item[]$(\leqq2)$ For any $\Phi$, if $\{X \in \mb{X} \mid  X \Vdash \Phi \}$ (we denote it as $ \set{\Phi}$) is not empty, then it has a unique $\leqq$-minimal element denoted by $\mini{\Phi}$.
\end{enumerate}

A descriptor revision $\circ$ on $K$ is \emph{based on} (or \emph{determined by}) some relational model $\relamod$ with respect to $K$ if and only if for any $\Phi$,

\begin{eqnarray}\nonumber
\mtod\footnotemark   \,\,\,\,\,\,\, K \circ \Phi=
   \begin{cases}
   \mini{\Phi} &\mbox{if $\set{\Phi}$ is not empty,}\\
   K &\mbox{otherwise.}
   \end{cases}
\end{eqnarray}

\footnotetext{\, Provided that $\relamod$ is a relational model, $\mb{X} $ is equivalent to the domain of $\leqq$ since $K \in \mb{X}$ and $K \leqq X$ for all $X \in \mb{X}$. So $\leqq$ in itself can represent the $\relamod$ faithfully.}

\end{DEF}

$\mb{X}$ could be seen as an \textit{outcome set} which includes all the potential outcomes under various belief change patterns. The ordering $\leqq$ (with the strict part $<$) brings out a direct-selection mechanism, which selects the final outcome among candidates satisfying a specific success condition. Given condition $(\leqq2)$, this sort of selection is achievable for any success condition satisfiable in $\mb{X}$. We call descriptor revision constructed in this way \emph{relational} descriptor revision.

As \cite{zhang_believability_2017} pointed out, in so far as the selection mechanism is concerned, descriptor revision is at a more abstract level comparing with the AGM revision. In the construction of descriptor revision $\circ$, ``it assumes that there exists an outcome set which contains all the potential outcomes of the operation $\circ$, but it says little about what these outcomes should be like'' \cite[p. 41]{zhang_believability_2017}. In contrast, in the AGM framework, the belief change is supposed to satisfy the principle of consistency preservation and the principle of the informational
economy \cite{gardenfors_knowledge_1988}. Therefore, the intersection step in the construction of belief change in the AGM framework becomes dispensable in the context of descriptor revision. This may explain why the descriptor revision model could be a select-direct approach.

\subsection{Choice revision constructed from descriptor revision}

The success condition for choice revision $\cro$ with a finite input  could be easily expressed by descriptor $\{ \mathfrak{B} \varphi_0 \vee \cdots \vee \mathfrak{B} \varphi_n \}$. So, it is straightforward to construct choice revision through descriptor revision as follows.

\begin{DEF}[\cite{hansson_descriptor_2017}]
Let $\circ$ be some descriptor revision. A choice revision $\cro$ on $K$ is \emph{based on} (or \emph{determined by}) $\circ$ if and only if for any finite set $A$, 

\begin{eqnarray}\nonumber
\dtoc   \,\,\,\,\,\,\, K \cro A= 
\begin{cases}
K \circ \{ \mathfrak{B} \varphi_0 \vee \cdots \vee \mathfrak{B} \varphi_n \} & \mbox{if $A = \{\varphi_0 , \cdots , \varphi_n \} \neq \emptyset$},\\
K & \mbox{otherwise}.
\end{cases}
\end{eqnarray}

\end{DEF}

Henceforth, we say $\cro$ is \emph{based on} (or \emph{determined by}) some relational model if it is based on the descriptor revision determined by the same model. The main purpose of this section is to investigate the formal properties of choice revision based on such models.

\subsection{Postulates and representation theorem}

It is observed that the choice revision determined by relational models should satisfy a set of arguably plausible postulates on choice revision. 

\begin{OBS}\label{Observation on the postulates satisfied by the choice revision constructed from relational model}
Let $\cro$ be a choice revision determined by any relational descriptor revision $\relamod$. Then it satisfies the following postulates: 
\begin{enumerate}
\item[] $\mathrm{(\cro 1)}$ $\cons{K \cro A} = K \cro A$. (\textbf{$\cro$-closure})
\item[] $\mathrm{(\cro 2)}$ $K \cro A = K$ or $A \cap (K \cro A) \neq \emptyset$. (\textbf{$\cro$-relative success})
\item[] $\mathrm{(\cro 3)}$ If $A \cap (K \cro B) \neq \emptyset$, then $A \cap (K \cro A) \neq \emptyset$. (\textbf{$\cro$-regularity}) 
\item[] $\mathrm{(\cro 4)}$ If $A \cap K \neq \emptyset$, then $K \cro A = K$. (\textbf{$\cro$-confirmation})
\item[] $\mathrm{(\cro 5)}$ If $(K \cro A) \cap B \neq \emptyset$ and $(K \cro B) \cap A \neq \emptyset$, then $K \cro A = K \cro B$. (\textbf{$\cro$-reciprocity})
\end{enumerate}
\end{OBS}

Moreover, another plausible condition on choice revision follows from this set of postulates. 

\begin{OBS}\label{Observation on cro-syntax irrelevance}
If $\cro$ satisfies $\cro$-closure, relative success, regularity and reciprocity, then $\cro$ satisfies:
\begin{enumerate}
\item[]  If $A \equiv B$, then $K \cro A = K \cro B$. (\textbf{$\cro$-syntax irrelevance})
\end{enumerate}
\end{OBS}

It is easy to see that the postulates in above are natural generalizations of the following postulates on sentential revision:

\begin{enumerate}
\item[] $\mathrm{(\ast 1)}$ $\cn(K \ast \varphi)=K \ast \varphi$ \textit{\textbf{($\ast$-closure)}}
\item[] $\mathrm{(\ast 2)}$  If $K \ast \varphi \neq K $, then $\varphi \in K \ast \varphi$ \textit{\textbf{($\ast$-relative success)}}
\item[] $\mathrm{(\ast 3)}$  If $\varphi \in K $, then $K \ast \varphi = K$ \textit{\textbf{($\ast$-confirmation)}}
\item[] $\mathrm{(\ast 4)}$  If $\psi \in K \ast \varphi$, then $\psi \in K \ast \psi$ \textit{\textbf{($\ast$-regularity)}}
\item[] $\mathrm{(\ast 5)}$  If $\psi \in K \ast \varphi$ and $\varphi \in K \ast \psi$, then $K\ast \varphi = K \ast \psi$ \textit{\textbf{($\ast$-reciprocity)\footnote{\, This postulate is first discussed in \cite{alchourron1982logic} in the context of maxichoice revision.}}}
\end{enumerate}
and 
\begin{enumerate}
\item[] If $\varphi \dashv \Vdash \psi$, then $K\ast \varphi = K \ast \psi$. \footnote{\, It is easy to check that $\ast$-extensionality is derivable from $(\ast 1)$, $(\ast 2)$, $(\ast 3)$ and ($\ast 5$).} \textit{\textbf{($\ast$-extensionality)}}
\end{enumerate}

The above postulates on choice revision are as intuitively plausible as their correspondents on sentential revision, except that the meaning of $\cro$-reciprocity seems not so transparent as that of $\ast$-reciprocity. However, given some weak conditions, we can show that the $\cro$-reciprocity postulate is equivalent to a more understandable condition as follows.

\begin{OBS}\label{Observation that reciprocity is equivalent to cautiousness}
Let choice operation $\cro$ satisfy $\cro$-relative success and \linebreak $\cro$-regularity. Then it satisfies $\cro$-reciprocity iff it satisfies:
\begin{enumerate}
\item[] If $A \subseteq B$ and $(K \cro B) \cap A \neq \emptyset$, then $K \cro A =K \cro B$. (\textbf{$\cro$-cautiousness})
\end{enumerate} 
\end{OBS}

\noindent The postulate $\cro$-cautiousness reflects a distinctive characteristic of choice revision modelled by relational models: The agent who performs this sort of belief change is cautious in the sense of only adding the smallest possible part of the new information to her original beliefs. It follows immediately from $\cro$-relative success and $\cro$-cautiousness that if $A \cap (K \cro A) \neq \emptyset$, then $K \cro A = K \cro \{\varphi\}$ for some $\varphi \in A$. Thus, it is not surprising that the following postulate follows.

\begin{OBS}\label{Observation that dichotomy can be derived from reciprocity}
If $\cro$ satisfies $\cro$-relative success, regularity and reciprocity, then $\cro$ satisfies:
\begin{enumerate}
\item[] $K \cro (A \cup B) = K \cro A$ or $K \cro (A \cup B) = K \cro B$. (\textbf{$\cro$-dichotomy})
\end{enumerate}
\end{OBS}

\noindent In contrast to $(\cro 1)$ through $(\cro 5)$, postulates $\cro$-cautiousness and $\cro$-dichotomy do not have directly corresponding postulates in the context of sentential revision. This suggests that though $(\cro 1)$ through $(\cro 5)$ naturally generalize $(\ast 1)$ through $(\ast 5)$, this sort of generalization is not so trivial as we may think of. As another evidence for this, the following observation shows that the properties of $(\ast 1)$ through $(\ast 5)$ and those of their generalizations are not always paralleled.

\begin{OBS}\label{Observation that reciprocity is equivalent to strong reciprocity}
Let $\cro$ satisfy $\cro$-regularity. Then it satisfies $\cro$-reciprocity iff it satisfies 
\begin{enumerate}
\item[] For any $n \geq 1 $, if $(K \cro A_1 ) \cap A_0 \neq \emptyset$, $\cdots$, $(K \cro A_{n} ) \cap A_{n-1} \neq \emptyset$, $(K \cro A_{0} ) \cap A_{n} \neq \emptyset$, then $K \cro A_0 = K \cro A_1 = \cdots = K \cro A_n$. (\textbf{$\cro$-strong reciprocity})
\end{enumerate}
\end{OBS}

\noindent $\ast$-strong reciprocity is a generalization of the following postulate on sentential revision:

\begin{enumerate}
\item[] For any $n \geq 1 $, if  $\varphi_0 \in K \star \varphi_{1}$, $\cdots$, $\varphi_{n-1} \in K \ast \varphi_n$ and $\varphi_n \in K \star \varphi_0 $, then $K \star \varphi_0 = K \star \varphi_2= \cdots = K \star \varphi_n$. (\textbf{$\ast$-strong reciprocity} )\footnote{\, $\ast$-strong reciprocity 
is closely related to a non-monotonic reasoning rule named ``loop'' which is first introduced  in \cite{kraus_nonmonotonic_1990}. For more discussion on this, see \cite{makinson_relations_1991}.}
\end{enumerate}

\noindent However, in contrast to the result in Observation \ref{Observation that reciprocity is equivalent to strong reciprocity}, $\ast$-strong reciprocity is not derivable from $(\ast 1)$ through $(\ast 5)$.\footnote{\, To see this, let $K = \conp{\taut}$ and revision operation $\ast$ on $K$ defined as: (i) if $p_0 \wedge p_1 \vdash \varphi$ and $\varphi \vdash p_0$, then $K \ast \varphi = \conp{p_0 \wedge p_1}$; (ii) if $p_1 \wedge p_2 \vdash \varphi$ and $\varphi \vdash p_1$, then $K \ast \varphi = \conp{p_1 \wedge p_2}$; (iii) if $p_0 \wedge p_2 \vdash \varphi$ and $\varphi \vdash p_2$, then $K \ast \varphi = \conp{p_0 \wedge p_2}$; (iv) otherwise, $K \ast \varphi = \conp{\varphi}$. It is easy to check that $\ast$ satisfies $(\ast 1)$ through $(\ast 5)$ but not $\ast$-strong reciprocity.}

After an investigation on the postulates $(\cro 1)$ through $(\cro 5)$ satisfied by choice revision based on rational models, the question raises naturally whether the choice revision could be axiomatically characterized by this set of postulates. We get a partial answer to this question: a representation theorem is obtainable when $\mathcal{L}$ is finite.

\begin{THE}\label{Representation theorem for choice revision derived from descriptor revision of finite language}
Let $\mathcal{L}$ be a finite language. Then, $\cro$ satisfies $(\cro 1)$ through $(\cro 5)$ iff it is a choice revision based on some relational model.
\end{THE}

\subsection{More properties of choice revision}

In this subsection, we will study additional properties of choice revision from the point of view of postulates. The postulates introduced in the previous subsection do not necessarily cover all the reasonable properties of this operation. In what follows we are going to investigate some additional ones, in particular,  the following:

\begin{enumerate}
\item[] If $A \neq \emptyset$, then$A \cap (K \cro A) \neq \emptyset$. (\emph{\textbf{$\cro$-success}})
\item[] If $A \not \equiv \{\falsum\}$, then $K \cro A \nvdash \falsum$. (\emph{\textbf{$\cro$-consistency}})
\end{enumerate}

To some extent, $\cro$-success is a strengthening of $\cro$-relative success and $\cro$-regularity, but it does not say anything about the limiting case in which the input is empty. To cover this limiting case, we need the following postulate:

\begin{enumerate}
\item[] If $A = \emptyset$, then $K \cro A = K$. (\textbf{$\cro$-vacuity})
\end{enumerate}

The interrelations among $\cro$-success, $\cro$-relative success and \linebreak $\cro$-regularity are summarized as follows.

\begin{OBS}\label{Observation on the inter-derivability among success, relative , vacuity and regularity }
Let $\cro$ be some choice revision on $K$.
\begin{enumerate}
\item If $\cro$ satisfies relative success, then it satisfies vacuity.
\item If $\cro$ satisfies success and vacuity, then it satisfies relative success.
\item If $\cro$ satisfies success, then it satisfies regularity.
\end{enumerate}
\end{OBS}

$\cro$-consistency is a plausible constraint on a rational agent. While accepting $\cro$-success and $\cro$-consistency as ``supplementary'' postulates for choice revision $\cro$, the corresponding relational model on which $\cro$ is based will also need to satisfy some additional properties. We use the following representation theorem to conclude this subsection.

\begin{THE}\label{Representation thoerem for choice revision additionally satisfying success and consistency }
Let $\mathcal{L}$ be a finite language and $\cro$ some revision operation on $K \subseteq \mathcal{L}$. Then, $\cro$ satisfies  $\cro$-closure, $\cro$- success, $\cro$-vacuity, $\cro$-confirmation, $\cro$-reciprocity and $\cro$-consistency iff 
 it is a choice revision determined by some relational model which satisfies the following two condition:
\begin{enumerate}
\item[] $(\mb{X} 3)$ $\conp{\falsum} \in \mb{X}$; 
\item[] $(\leqq 3)$ $\set{\mathfrak{B} \varphi} \neq \emptyset$ and $\mini{\mathfrak{B} \varphi } < \conp{\falsum}$ for every $\varphi \not \vdash \falsum$.
\end{enumerate}
\end{THE}

\section{An alternative modelling for choice revision}\label{section on an alternative modelling for choice revision}

In this section, we propose an alternative modelling for choice revision, which is based on so-called \textit{multiple believability relations}. A believability relation $\preceq$ is a binary relation on sentences of $\mathcal{L}$. Intuitively, $\varphi \preceq \psi$ means that the subject is at least as prone to believing $\varphi$ as to believing $\psi$.\footnote{\, For more detailed investigation on believability relations, including its relationship with the epistemic entrenchment relation introduced in \cite{gardenfors_revisions_1988}, see \cite{hansson_relations_2014} and \cite{zhang_believability_2017}.} We can generalize $\preceq$ to a multiple believability relation $\preceq_{\ast}$ which is a binary relation on the set of all finite subsets of $\mathcal{L}$ satisfying:

\begin{enumerate}
\item[] \abtob \,\,\,\,\,\,$\varphi \preceq \psi$ iff $\{\varphi\} \preceq_{\ast}  \{\psi\}$.
\end{enumerate}

\noindent This kind of generalization can be done in different ways, and at least two distinct relations can be obtained, namely \textit{package multiple believability relations}, denoted by $\preceqp$, and \textit{choice multiple believability relations}, denoted by $\preceqc$ (with symmetric part $\simeqc$ and strict part $\precc$). Intuitively, $A \preceqp B$ means that it is easier for the subject to believe all propositions in $A$ than to believe all propositions in $B$ and $A \preceqc B$ means that it is easier for the subject to believe some proposition in A than to believe some proposition in $B$. 

\noindent $\preceqp$ is of little interest since $A \preceqp B$ can be immediately reduced to $\& A \preceq \& B$, given that $A$ and $B$ are finite. In what follows, multiple believability relations (or multi-believability relations for short) only refer to choice multiple believability relations $\preceqc$. ($\{\varphi\} \preceqc A$ will be written as $\varphi \preceqc A$ for simplicity.) 

\subsection{Postulates on multi-believability relations}

Recall the following postulates on believability relations $\preceq$ introduced in \cite{zhang_believability_2017}:
\begin{enumerate}
\item[] \emph{\textbf{$\preceq$-transitivity}}: If $\varphi \preceq \psi$ and $\psi \preceq \lambda$, then $\varphi \preceq \lambda$. 
\item[] \emph{\textbf{$\preceq$-weak coupling}}: If $\varphi \simeq \varphi \wedge \psi $ and $\varphi \simeq \varphi \wedge \lambda$, then $\varphi \simeq \varphi  \wedge (\psi \wedge \lambda)$.
\item[] \emph{\textbf{$\preceq$-coupling}}: If $\varphi \simeq \psi$, then $\varphi \simeq \varphi \wedge \psi$.
\item[] \emph{\textbf{$\preceq$-counter dominance}}: If $\varphi \vdash \psi$, then $\psi \preceq \varphi$.
\item[] \emph{\textbf{$\preceq$-minimality}}: $\varphi \in K$ if and only if $\varphi \preceq \psi$ for all $\psi$.
\item[] \emph{\textbf{$\preceq$-maximality}}: If $\psi \preceq \varphi$ for all $\psi$, then $\varphi \equiv \falsum$.
\item[] \emph{\textbf{$\preceq$-completeness}}: $\varphi \preceq \psi$ or $\psi \preceq \varphi$
\end{enumerate} 

Transitivity is assumed for almost all orderings. In virtue of the intuitive meaning of believability relation, $\varphi \simeq \varphi \wedge \psi$ represents that the agent will accept $\psi$ in the condition of accepting $\varphi$. Thus, the rationale for $\preceq$-weak coupling is that if the agent will consequently add $\psi$ and $\lambda$ to her beliefs when accepting $\varphi$, then she also adds the conjunction of them to her beliefs in this case. This is reasonable if we assume that the beliefs of the agent are closed under the consequence operation. The justification of $\preceq$-counter dominance is that if $\varphi$ logically entails $\psi$, then it will be a smaller change and hence easier for the agent to accept $\psi$ rather than to accept $\varphi$, because then $\psi$ must be added too, if we assume that the beliefs of the agent are represented by a belief set. $\preceq$-coupling is a strengthening of $\preceq$-weak coupling.\footnote{\, It is easy to see that $\preceq$-coupling implies $\preceq$-weak coupling, provided that $\preceq$-transitivity and $\preceq$-counter dominance hold. } It says that if $\varphi$ is equivalent to $\psi$ in believability, then the agent will consequently add $\psi$ to her beliefs in case of accepting $\varphi$ and vice versa. $\preceqc$-minimality is justifiable since nothing needs to be done to add $\varphi$ to $K$ if it is already in $K$. $\preceq$-maximality is justifiable since it is reasonable to assume that it is strictly more difficult for a rational agent to accept $\falsum$ than to accept any non-falsum. $\preceq$-completeness seems a little bit strong. It says that all pairs of sentences are comparable in believability. In accordance with \cite{zhang_believability_2017}, we call relations satisfying all these postulates \textit{quasi-linear believability relations}.

We can generalize these postulates on believability relations in a natural way to postulates multi-believability relations as follows:\footnote{\, In what follows, it is always assumed that all sets $A$ and $B$ and $C$ mentioned in postulates on multi-believability relations are finite sets.}

\begin{enumerate}
\item[] \emph{\textbf{$\preceqc$-transitivity}}: If $A \preceqc B$ and $B \preceqc C$, then $A \preceqc C$.
\item[] \emph{\textbf{$\preceqc$-weak coupling}}: If \( {A \simeqc A \owedge B} \) and \( {A \simeqc A \owedge C} \), then $A \simeqc A \owedge B \owedge C$.
\item[] \emph{\textbf{$\preceqc$-coupling}}: If $A \simeqc B$, then $A \simeqc A \owedge B$. 
\item[] \emph{\textbf{$\preceqc$-counter dominance}}: If for every $\varphi \in B$ there exists $\psi \in A$ such that $\varphi \vdash \psi$, then $A \preceqc B$. 
\item[] \emph{\textbf{$\preceqc$-minimality}}: $A \preceqc B$ for all $B$ if and only if $A \cap K \neq \emptyset$. 
\item[] \emph{\textbf{$\preceqc$-maximality}}: If $B$ is not empty and $A \preceqc B$ for all non-empty $A$, then $B \equiv \{\falsum\}$. 
\item[] \emph{\textbf{$\preceqc$-completeness}}: $A \preceqc B$ or $B \preceqc A$.
\end{enumerate}

\noindent These postulates on multi-believability relations can be understood in a similar way that their correspondents on believability relations are understood. 

Furthermore, we propose the following two additional postulates on multi-believability relations:

\begin{enumerate}
\item[] \emph{\textbf{$\preceqc$-determination}}: $A \precc \emptyset$ for every non-empty $A$.
\item[] \emph{\textbf{$\preceqc$-union}}: $A \preceq A \cup B$ or $B \preceq A \cup B$.
\end{enumerate}

\noindent At least on the surface, these two could not be generalizations of any postulate on believability relation. In some sense the meaning of \linebreak $\preceqc$-determination is correspondent to that of $\cro$-success, since if it is a \linebreak strictly smaller change for the agent to accept some sentences from a non-empty $A$ rather than to take some sentences from the empty set, which is obviously impossible, then it seems to follow that the agent will successfully add some sentences in $A$ to her original beliefs when exposed to the new information represented by $A$, and vice versa. Similarly, there is an obvious correspondence between the forms and meanings of $\preceqc$-union and $\cro$-dichotomy. They both suggest that to partially accept a non-empty $A$ is equivalent to accept some single sentence in $A$. This is plausible if we assume that the agent is extremely cautious to the new information.

\begin{OBS}\label{Observation on the interderivability among postulates on multiple believability relations}
Let $\preceqc$ be some multi-believability relation satisfying $\preceqc$-transitivity and $\preceqc$-counter dominance. If it satisfies $\preceqc$-union in addition, then
\begin{enumerate}
\item It satisfies $\preceqc$-completeness.
\item It satisfies $\preceqc$-weak coupling iff it $\preceqc$-satisfies coupling.
\end{enumerate} 
\end{OBS}

\noindent Observation \ref{Observation on the interderivability among postulates on multiple believability relations} indicates that $\preceqc$-union is strong. It should be noted that for a believability relation, neither $\preceq$-completeness nor $\preceq$-coupling can be derived from $\preceq$-transitivity, $\preceq$-counter dominance and $\preceq$-weak coupling.

In what follows, we name multi-believability relations satisfying all the above postulates \textit{standard multi-believability relations}. 

\subsection{Translations between believability relations and multiple believability relations}
In this subsection, we will show that although it is impossible to find a postulate on believability relations that corresponds to $\preceqc$-determination or $\preceqc$-union, there exists a translation between quasi-linear believability relations and standard multi-believability relations.

\begin{OBS}\label{observation for reduction to single sentence}
Let $\preceqc$ satisfy $\preceqc$-determination, $\preceqc$-transitivity and $\preceqc$-counter dominance. Then, for any non-empty finite sets $A$ and $B$,
\begin{enumerate}
\item $A \preceqc B$ if and only if there exists $\varphi \in A$ such that $\varphi \preceqc B$.
\item $A \preceqc B$ if and only if $A \preceqc \varphi$ for all $\varphi \in B$.
\end{enumerate}
\end{OBS}

\noindent This observation suggests that $\preceq$ and $\preceqc$ can be linked through the following two transitions:
\begin{enumerate}
\item[] \mbtob \,\,\,\,\,\,$\varphi \preceq \psi$ iff $\{\varphi\} \preceqc  \{\psi\}$.
\item[] \btomb \,\,\,\,\,\,$A \preceqc B$ iff $B = \emptyset$ or there exists $\varphi \in A$ such that $\varphi \preceq \psi$ for every $\psi \in B$.
\end{enumerate}

\noindent This is confirmed by the following theorem.

\begin{THE}\label{Theorem on the translation between single and multi believability relations}
\begin{enumerate}
\item If $\preceq$ is a quasi-linear believability relation and $\preceqc$ is constructed from $\preceq$ through the way of \btomb, then $\preceqc$ is a standard multi-believability relation and $\preceq$ can be retrieved from $\preceqc$ in the way of \mbtob.
\item If $\preceqc$ is a standard multi-believability relation and $\preceq$ is constructed from $\preceqc$ through \mbtob, then $\preceq$ is a quasi-linear believability relation and $\preceqc$ can be retrieved from $\preceq$ through \btomb.
\end{enumerate}
\end{THE}

\subsection{Choice revision constructed from multi-believability relations}

Now we turn to the construction of choice revision through \linebreak multi-believability relations. Recall that a sentential revision $\ast$ can be constructed from a believability relation $\preceq$ in this way \cite{zhang_believability_2017}:
\begin{eqnarray}\nonumber
\rtoo \,\,\,\,\,\,\,K \ast \varphi= \{ \psi \mid \varphi \simeq \varphi \wedge \psi \} 
\end{eqnarray} 

\noindent As we have explained, $\varphi \simeq \varphi \wedge \psi$ could be understood as that the agent will consequently accept $\psi$ in case of accepting $\varphi$. So, the set $\{ \psi \mid \varphi \simeq \varphi \wedge \psi \} $ is just the agent's new set of beliefs after she performed belief revision with input $\varphi$. Thus, we can similarly construct choice revision from multi-believability relations in the following way:

\begin{DEF}\label{definition of choice revision based on multi-believability relations}
Let $\preceqc$ be some multi-believability relation. A choice revision $\cro$ on $K$ is based on (or determined by) $\preceqc$ if and only if: for any finite $A$, 
\begin{eqnarray}\nonumber
\mbtoc   \,\,\,\,\,\,\, K \cro A= 
\begin{cases}
\{\varphi \mid A \simeqc A \owedge \varphi \} & \mbox{If $A \precc \emptyset$},\\
K & \mbox{otherwise}.
\end{cases}
\end{eqnarray}
\end{DEF}

The primary results of this section are the following two representation theorems. Comparing with Theorems \ref{Representation theorem for choice revision derived from descriptor revision of finite language} and \ref{Representation thoerem for choice revision additionally satisfying success and consistency }, these two theorems are applicable to more general cases since they do not assume that the language $\mathcal{L}$ is finite. These two theorems demonstrate that multi-believability relations provide a fair modelling for choice revision characterized by the set of postulates mentioned in Section \ref{section choice revison based on descriptor revision}.

\begin{THE}\label{Representation theorem for choice revision based on weak multiple believability relation}
Let $\cro$ be some choice revision on $K$. Then, $\cro$ satisfies $(\cro 1)$ through $(\cro 5)$ iff it is determined by some multi-believability relation $\preceqc$ satisfying $\preceqc$-transitivity, $\preceqc$-weak coupling, $\preceqc$-counter-dominance, $\preceqc$-minimality \linebreak and $\preceqc$-union.
\end{THE}

\begin{THE}\label{Representation theorem for choice revision based on standard multiple believability relation}
Let $\cro$ be some choice revision on $K$. Then, $\cro$ satisfies  $\cro$-closure, $\cro$- success, $\cro$-vacuity, $\cro$-confirmation, $\cro$-reciprocity and $\cro$-consistency iff it is determined by some standard multi-believability relation.
\end{THE}

Considering the translation between multi-believability relations and believability relations (Theorem \ref{Theorem on the translation between single and multi believability relations}), it seems that these results can be easily transferred to the context of believability relations. However, if we drop some postulates on multi-believability relation such as $\preceqc$-determination, the translation between multi-believability relation and believability relation will not be so transparent, at least it will not be so straightforward as \btomb{ }and \mbtob. As a consequence, the result in Theorem \ref{Representation theorem for choice revision based on weak multiple believability relation} may not be possible to transfer to believability relations in a straightforward way. Moreover, comparing with postulates on believability relations, postulates on multi-believability relations such as $\preceqc$-determination and $\preceqc$-union can present our intuitions on choice revision in a more direct way. Thus, the multi-believability relation is still worth to be studied in its own right.

\section{Conclusion and future work}\label{section conclusion}

As a generalization of traditional belief revision, choice revision has more realistic characteristics. The new information is represented by a set of sentences and the agent could partially accept these sentences as well as reject the others. From the point of technical view, choice revision is interesting since the standard ``select-and-intersect'' methodology in modellings for belief change is not suitable for it. But instead, it can be modelled by a newly developed framework of descriptor revision, which employs a ``select-direct'' approach. After reviewing the construction of choice revision in the framework of descriptor revision, under the assumption that the language is finite, we provided two sets of postulates as the axiomatic characterizations for two variants of choice revision based on such constructions  (in Theorem \ref{Representation theorem for choice revision derived from descriptor revision of finite language} and \ref{Representation thoerem for choice revision additionally satisfying success and consistency }). These postulates, in particular, \linebreak $\cro$-cautiousness and $\cro$-dichotomy, point out that choice revision modelled by descriptor revision has the special characteristic that the agent who performs this sort of belief change is cautious in the sense that she only accepts the new information to the smallest possible extent.

For AGM revision and contraction, there are various independently motivated modellings which are equivalent in terms of expressive power. In this contribution, we also propose an alternative modelling for choice revision. We showed that multi-believability relations can also construct the choice revision axiomatically characterized by the sets of postulates proposed for choice revision based on descriptor revision (Theorem \ref{Representation theorem for choice revision based on weak multiple believability relation} and \ref{Representation theorem for choice revision based on standard multiple believability relation}). Moreover, these results are obtainable without assuming that the language is finite. This may indicate that multi-believability relations are an even more suitable modelling for choice revision.

The study in this contribution can be developed in at least three directions. First, the cautiousness constraint on choice revision, reflected by $\cro$-cautiousness, certainly could be loosened. We think it is an interesting topic for future work to investigate the modeling and axiomatic characterization of more ``reckless'' variants of choice revision. Secondly, as it was showed in \cite{zhang_believability_2017} that AGM revision could be reconstructed from believability relations satisfying certain conditions, it is interesting to ask which conditions a multi-believability relation should satisfy so that its generated choice revision coincides with an  AGM revision when the inputs are limited to singletons. Finally, it is technically interesting to investigate choice revisions with an infinite input set, though they are epistemologically unrealistic.

\section*{Appendix: Proofs}\label{appendix}

\begin{LEM}\label{Lemma on the representation element}
Let $\preceqc$ be some multiple believability relation which satisfies \linebreak $\preceqc$-counter~dominance and $\preceqc$-transitivity. Then, 
\begin{enumerate}
\item If $\preceqc$ satisfies $\preceqc$-union, then for every non-empty $A$, there exists some $\varphi \in A$ such that $\varphi \simeqc A$.
\item For every $\varphi \in A$, $A \simeqc A \owedge \varphi$ if and only if $\varphi \simeqc A$.
\end{enumerate}
\end{LEM}

\begin{proof}[Proof for Lemma \ref{Lemma on the representation element}:]
\textit{1.} We prove this by mathematical induction on the size $n$ ($n \geq 1$) of $A$. Let $n = 1$, then it follows immediately. Suppose hypothetically that it holds for $n = k$ ($k \geq 1$). Let $n = k+1$. Since $k \geq 1$, there exists a non-empty set $B$ containing $k$ elements and a sentence $\varphi$ such that $A = B \cup \{\varphi\}$. By $\preceqc$-counter dominance and $\preceqc$-union, (i) $A \simeqc \{\varphi\}$ or (ii) $A \simeqc B$. The case of (i) is trivial. In the case of (ii), by the hypothetical supposition, there exists some $\psi \in B \subseteq A$ such that $A \simeqc B \simeqc \psi$. So, by $\preceqc$-transitivity, $A \simeqc \varphi$. To sum up (i) and (ii), there always exists some $\varphi \in A$ such that $\varphi \simeqc A$.\\
\textit{2. From left to right:} Let $\varphi \in A $ and $A \owedge \varphi \simeqc A$. By $\preceqc$-counter dominance, $A \preceqc \varphi$ and $ \varphi \preceqc A \owedge \varphi$. And it follows from $\varphi \preceqc A \owedge \varphi$ and $A \owedge \varphi \simeqc A$ that $\varphi \preceqc A$ by $\preceqc$-transitivity. Thus, $\varphi \simeqc A$. \textit{From right to left:} Let $\varphi \in A$ and $\varphi \simeq A$. By $\preceqc$-counter-dominance, $A \owedge \varphi \preceqc \varphi$. So $A \owedge \varphi \preceqc A$ by $\preceqc$-transitivity. Moreover, $A \preceqc A \owedge \varphi$ by $\preceqc$-counter-dominance. Thus, $A \owedge \varphi \simeqc A$.
\end{proof}

\begin{proof}[Proof for Observation \ref{Observation on the postulates satisfied by the choice revision constructed from relational model}:]
It is easy to see that $\cro$ satisfies $\cro$-closure and $\cro$-relative success. We only check the remaining three postulates. We let $\Belsome{A}$ denote the descriptor $ \{ \mathfrak{B} \varphi_0 \vee \cdots \vee \mathfrak{B} \varphi_n\}$ when $A = \{ \varphi_0 , \cdots , \varphi_n\} \neq \emptyset$.\\
\textit{$\cro$-regularity:} Let $(K \cro B) \cap A \neq \emptyset$. It follows that $A \neq \emptyset$ and $\set{\Belsome{A}} \neq \emptyset$. So $K \cro A = \mini{\Belsome{A}}$ by the definition of $\cro$. Thus, $(K \cro A) \cap A \neq \emptyset$.\\
\textit{$\cro$-confirmation:} Let $A \cap K \neq \emptyset$. Then $A \neq \emptyset $ and $K \in \set{\Belsome{A}}$. It follows from $(\leqq 1)$ and $(\leqq 2)$ that $K$ is the unique $\leqq$-minimal element in $\set{\Belsome{A}}$. Thus, $K \cro A = \mini{\Belsome{A}} = K$.\\ 
\textit{Reciprocity:} Let $(K \cro A_0 ) \cap A_1 \neq \emptyset$ and $(K \cro A_{1} ) \cap A_{0} \neq \emptyset$. Let $i \in \{0,1\}$. It follows that $A_i \neq \emptyset $ and $\set{\Belsome(A_i)} \neq \emptyset$ and hence $K \cro A_i = \mini{\Belsome(A_i)}$ by the definition of $\cro$. So it follows from $\mini{\Belsome{A_0}} \cap A_1 \neq \emptyset$ and $\mini{\Belsome{A_1}} \cap A_{0} \neq \emptyset$ that $\mini{\Belsome{A_0}} \in \set{\Belsome{A_1}}$ and $\mini{\Belsome{A_1}} \in \set{\Belsome{A_0}}$ and hence $\mini{\Belsome{A_0}} \leqq \mini{\Belsome{A_1}} \leqq  \mini{\Belsome{A_0}}$ by $(\leqq 2)$. Since the minimal element in $\set{\Belsome(A_0)}$ is unique by $(\leqq 2)$, it follows that $\mini{\Belsome{A_0}} = \mini{\Belsome{A_1}}$, i.e. $ K \cro A_{0} = K \cro A_{1}$.
\end{proof}

\begin{proof}[Proof for Observation \ref{Observation on cro-syntax irrelevance}:]
Let $A \equiv B$. Suppose $A \cap (K \cro A) = \emptyset$, then $A \cap (K \cro B) = \emptyset$ due to $\cro$-regularity. Hence, $B \cap (K \cro B) = \emptyset$ by $\cro$-closure. It follows that $K \cro A = K \cro B = K$ by $\cro$-relative success. Suppose $A \cap (K \cro A) \neq \emptyset$, then $B \cap (K \cro A) \neq \emptyset$ by $\cro$-closure, so $B \cap (K \cro B) \neq \emptyset$ by $\cro$-regularity, and hence $A \cap (K \cro B) \neq \emptyset$ by $\cro$-closure. It follows that $K \cro A = K \cro B$ by $\cro$-reciprocity. Thus, $\cro$ satisfies syntax irrelevance in any case.
\end{proof}

\begin{proof}[Proof for Observation \ref{Observation that reciprocity is equivalent to cautiousness}:]
\textit{From left to right:} Let $A \subseteq B$ and $(K \cro B) \cap A \neq \emptyset$. Then, $A \neq \emptyset $ and hence $ (K \cro A) \cap A \neq \emptyset$ by $\cro$-regularity. Since $A \subseteq B$, it follows that $(K \cro A) \cap B \neq \emptyset$. Thus, $K \cro A =K \cro B$ by $\cro$-reciprocity.\\ 
\textit{From right to left:} Let $A \cap (K \cro B) \neq \emptyset$ and $B \cap (K \cro A) \neq \emptyset$. It follows that $(A \cup B) \cap (K \cro B) \neq \emptyset$. By $\cro$-regularity, it follows that $(A \cup B) \cap (K \cro (A \cup B)) \neq \emptyset$. So $A \cap  (K \cro (A \cup B)) \neq \emptyset$ or $B \cap  (K \cro (A \cup B)) \neq \emptyset$. Without loss of generality, let $A \cap  (K \cro (A \cup B)) \neq \emptyset$, then $K \cro A = K \cro (A \cup B)$ by $\cro$-cautiousness. It follows that $B \cap (K \cro (A \cup B)) = B \cap (K \cro A) \neq \emptyset$ and hence $K \cro B = K \cro (A \cup B)$ by $\cro$-cautiousness. So $K \cro A = K \cro (A \cup B) = K \cro B$.
\end{proof}

\begin{proof}[Proof for Observation \ref{Observation that dichotomy can be derived from reciprocity}:]
Suppose $(A \cup B) \cap (K \cro (A\cup B)) = \emptyset$, then $(A \cup B) \cap (K \cro A) = (A \cup B) \cap (K \cro B)=  \emptyset$ by $\cro$-regularity. So $A \cap (K \cro A) = B \cap (K \cro B) = \emptyset$ and hence $K \cro (A \cup B) = K \cro A = K \cro B = K$ by $\cro$-relative success. Suppose $(A \cup B) \cap (K \cro (A\cup B)) \neq \emptyset$, then $A \cap  (K \cro (A\cup B)) \neq \emptyset $ or $B \cap  (K \cro (A\cup B)) \neq \emptyset$. Let $A \cap  (K \cro (A\cup B)) \neq \emptyset $, then $A \cap (K \cro A) \neq \emptyset$ by $\cro$-regularity and hence $(A \cup B) \cap (K \cro A) \neq \emptyset$. It follows that $K \cro (A \cup B) = K \cro A$ by $\cro$-reciprocity. Similarly, we can show that $K \cro (A \cup B) = K \cro B$ holds in the case of $B \cap  (K \cro (A\cup B)) \neq \emptyset$. Thus, $\cro$ satisfies $\cro$-dichotomy in any case.
\end{proof}

\begin{proof}[Proof for Observation \ref{Observation that reciprocity is equivalent to strong reciprocity}:]
\textit{From right to left:} It follows immediately.\\
\textit{From left to right:}  Assume $(\star)$ that  $(K \cro A_1 ) \cap A_0 \neq \emptyset$, $\cdots$, $(K \cro A_n ) \cap A_{n-1} \neq \emptyset$ and $(K \cro A_{0} ) \cap A_{n} \neq \emptyset$ for some $n \geq 1$. We prove that $K \cro A_0 = K \cro A_1 = \cdots = K \cro A_n$ by mathematical induction on $n$. For $n =1$, this follows immediately from $\cro$-reciprocity. Let us hypothetically suppose that it holds for $n = k$ ($k \geq 1$), then we should show that it also holds for $n = k+1$.\\
Let $A = \bigcup_{0 \leq i \leq k+1} A_{i}$. It follows from $(\star)$ that $A \cap (K \cro A_i) \neq \emptyset$ for every $0 \leq i \leq k+1$. So $A \cap (K \cro A) \neq \emptyset$ by $\cro$-regularity. It follows that there exists some $j$ with $0 \leq j \leq k+1$ such that $A_j \cap (K \cro A) \neq \emptyset$. Moreover, according to $(\star)$, if $j = 0$ then $A_{k+1} \cap (K \cro A_j) \neq \emptyset$ else $A_{j-1} \cap (K \cro A_j) \neq \emptyset$. It follows that $A \cap (K \cro A_j) \neq \emptyset$ in any case. So $K \cro A_j = K \cro A$ by $\cro$-reciprocity. Hence, as $A \cap (K \cro A_i) \neq \emptyset$ for every $0 \leq i \leq k+1$, it follows from $(\star)$ and $\cro$-reciprocity that if $0<j \leq k+1$, then $K \cro A_j =K \cro A = K \cro A_{j-1}$ else $K \cro A_j =K \cro A = K \cro A_{k+1}$. In each case, the length of the loop is reduced to $k$. So, it follows from the hypothetical supposition that $K \cro A_0 = K \cro A_1 = \cdots = K \cro A_{k+1}$. Thus, $\cro$ satisfies strong reciprocity.
\end{proof}

\begin{proof}[Proof for Theorem \ref{Representation theorem for choice revision derived from descriptor revision of finite language}:]
\textit{From construction to postulates:} See Observation \ref{Observation on the postulates satisfied by the choice revision constructed from relational model}.\\
\textit{From postulates to construction:} Let $\mb{X}=\{ K\cro A \mid A \subseteq \mathcal{L} \mbox{ and } A \mbox{ is finite}\}$. Let $\leqq^{\prime}$ be a relation on $\mb{X}$ defined as $X \leqq^{\prime} Y$ iff there exist elements $A_0$, $\cdots$, $A_n$ of $\mathcal{L}$ such that $X = K \cro A_0$, $Y = K \cro A_n$ and $(K \cro A_1 ) \cap A_0 \neq \emptyset$, $\cdots$, $(K \cro A_n ) \cap A_{n-1} \neq \emptyset$. We first show that $\leqq^{\prime}$ is a partial order:\\
\textit{Reflexivity:} Let $X = K \cro A$. If $X =K$, since $K$ is belief set, $X = K \cro \{\taut\}$ by $\cro$-confirmation. Moreover, by $\cro$-closure, $\taut \in K \cro \{\taut\}$. So $X \leqq^{\prime} X$.  If $X \neq K$, then $A \cap (K \cro A) \neq \emptyset$ by relative success. It follows immediately that $X \leqq^{\prime} X$. Thus, $X \leqq^{\prime} X$ holds for every $X \in \mb{X}$.\\
\textit{Transitivity:} It follows immediately from the definition of $\leqq^{\prime}$.\\
\textit{Anti-symmetry:} By Observation \ref{Observation that reciprocity is equivalent to strong reciprocity}, $\cro$ satisfies $\cro$-strong reciprocity. It follows immediately from this and the definition of $\leqq^{\prime}$ that $\leqq^{\prime}$ is anti-symmetric.\\
So, given the axiom of choice, there exists a linear order $\leqq$ such that $\leqq^{\prime} \subseteq \leqq$.\footnote{See \cite{jech_axiom_2008}.} We will show that $(\mb{X}, \leqq)$ is the relational model we are looking for.\\
$(\mb{X} 1)$: It is immediate from $\cro$-closure.\\
$(\mb{X} 2)$: It is immediate from that $K \cro \emptyset \in \mb{X}$ and $\cro$ satisfies $\cro$-relative success.\\
$(\leqq 1)$: Since $K$ is a belief set, $K = K \cro \{\taut\} $ by $\cro$-confirmation. Moreover, by $\cro$-closure, $\taut \in X$ for every $X \in \mb{X}$. So $K \leqq^{\prime} X$ for for every $X \in \mb{X}$. Thus, as $\leqq^{\prime}  \subseteq \leqq$, $K \leqq X$ for every $X \in \mb{X}$.\\
$(\leqq 2)$: Since $\mathcal{L}$ is finite, it is easy to see that the quotient of its power set under the equivalence relation $\equiv$ is finite. Moreover, by Observation \ref{Observation on cro-syntax irrelevance}, $\cro$ satisfies syntax-irrelevance. It follows that $\mb{X}$ is finite.
And as we have proved, $\leqq$ is a linear order. So $\leqq$ is well-ordered and hence $(\leqq 2)$ holds.\\
In order to show that $\cro$ is based on $\relamod$, we need to consider two cases:\\
(i) $K \cro A =K$: In this case, we just need to prove that if $A \neq \emptyset$ and $\set{\Belsome{A}} \neq \emptyset$, then $K \in \set{\Belsome{A}}$. It follows from $A \neq \emptyset$ and $\set{\Belsome{A}} \neq \emptyset$ that there exists some $B$ such that $A \cap (K \cro B )\neq \emptyset$. So $A \cap (K \cro A) \neq \emptyset$ by $\cro$-regularity. So, $A \cap K \neq \emptyset$ and hence $K \in \set{\Belsome{A}}$.\\ 
(ii) $K \cro A \neq K$: In this case, we only need to show that $\set{\Belsome{A}} \neq \emptyset$ and $K \cro A = \mini{\Belsome{A}}$. Since $K \cro A \neq K$, $(K \cro A) \cap A \neq \emptyset$ by $\cro$-relative success.  So $\set{\Belsome{A}} \neq \emptyset$ and $K \cro A \leqq^{\prime} X$ for every $X \in \set{\Belsome{A}}$. This means that $K \cro A \leqq X$ for every $X \in \Belsome{\Phi(A)}$ since $\leqq^{\prime} \subseteq \leqq$. Moreover, as we have shown, $\leqq$ is a linear order. Thus, $K \cro A = \mini{\Belsome{A}}$.\\
To sum up (i) and (ii), $ \cro $ is based on $\relamod$.
\end{proof}

\begin{proof}[Proof for Observation \ref{Observation on the inter-derivability among success, relative , vacuity and regularity }:]
\textit{1.} Let $A \neq \emptyset$, then $A \cap (K \cro A) = \emptyset$. So, by $\cro$-relative success, $K \cro A = K$.\\
\textit{2.} Let $A \cap (K \cro A) = \emptyset$. Then, by $\cro$-success, it follows that $A = \emptyset$. So, by vacuity, $K \cro A = K$.\\
\textit{3.} Let $A \cap (K \cro B) \neq \emptyset$, then $A \neq \emptyset$. So, by $\cro$-success, $A \cap (K \cro A) \neq \emptyset$.
\end{proof}

\begin{proof}[Proof for Theorem \ref{Representation thoerem for choice revision additionally satisfying success and consistency }:]
\textit{From right to left:} We only need to check success, vacuity and consistency.\\
$\cro$-success: It follows immediately from $(\mb{X} 3)$.\\
$\cro$-vacuity: It follows immediately from \mtod{ }and \dtoc.\\
$\cro$-consistency: It follows immediately from $(\leqq 3)$.
\\
\textit{From left to right:} Let $\relamod$ be defined in the same way as the relational model constructed in the proof of Theorem \ref{Representation theorem for choice revision derived from descriptor revision of finite language}. By Observation \ref{Observation on the inter-derivability among success, relative , vacuity and regularity }, $\cro$ satisfies all postulates listed in Theorem \ref{Representation theorem for choice revision derived from descriptor revision of finite language}. So it is also true that $\relamod$ constructed in this way is indeed a relational model from which a choice revision $\cro$ can be derived. 
In order to complete the proof, we only need to check that $\relamod$ satisfies $(\mb{X} 3)$ and $(\leqq 3)$: Since $\cro$ satisfies $\cro$-success and $\cro$-closure, $K \cro \{\falsum \} = \conp{\falsum} \in \mb{X}$, i.e. $(\mb{X} 3)$ holds of $\relamod$. Moreover, since $\cro$ satisfies $\cro$-consistency, $K \cro \{\varphi \} \neq \conp{\falsum}$ for every $\varphi \not \equiv \falsum$. It follows that $K \cro \{\varphi\} = \mini{\mathfrak{B} \varphi } < \conp{\falsum}$. So, $(\leqq 3)$ also holds of $\relamod$. 
\end{proof}

\begin{proof}[Proof for Observation \ref{Observation on the interderivability among postulates on multiple believability relations}:]
\textit{1.} It follows from $\preceqc$-counter dominance that $A \cup B \preceqc A$ and $A \cup B \preceqc B$. Moreover, by $\preceqc$-union, $A \preceqc A\cup B $ or $B \preceqc A \cup B$. So $A \preceqc B$ or $B \preceqc A$ by $\preceqc$-transitivity.\\
\textit{2. From left to right:} We first prove that $\preceqc$-coupling holds for all singletons, i.e. if $\varphi \simeqc \psi$, then $\varphi \simeqc \varphi \wedge \psi$. Let $\varphi \simeqc \psi$ and $A = \{\varphi, \psi\}$. Since it is immediate from $\preceqc$-counter dominance that $\varphi \preceqc \varphi \wedge \psi$, we only need to show $\varphi \wedge \psi \preceqc \varphi$. By the first item of Lemma \ref{Lemma on the representation element} and $\preceqc$-transitivity, $A \simeqc \varphi$ and $A \simeqc \psi$. So, by the second item of Lemma \ref{Lemma on the representation element}, $A \simeqc A \owedge \varphi$ and $A \simeqc A \owedge \psi$. So, by $\preceqc$-weak coupling, $A \simeqc A \owedge (\varphi \wedge \psi)$. By $\preceqc$-counter dominance, $\varphi \wedge \psi \preceqc A \owedge (\varphi \wedge \psi)$. So $\varphi \wedge \psi \preceqc A \owedge (\varphi \wedge \psi) \simeqc A \simeqc \varphi$ and hence $\varphi \wedge \psi \preceqc \varphi$ by $\preceqc$-transitivity.\\
Now we prove that $\preceqc$-coupling holds in general. Let $A \simeqc B$. If $A = \emptyset$, it follows immediately from $\preceqc$-counter-dominance that $\emptyset = A \simeqc A \owedge B = \emptyset$. If $B = \emptyset$, it follows from $A \simeqc B = \emptyset$ that $A \simeqc A \owedge B = \emptyset$. If $A \neq \emptyset$ and $B \neq \emptyset$, then there exist $\varphi \in A$ and $\psi \in B$ such that $\varphi \simeqc A$ and $\psi \simeqc B$ by the first item of Lemma \ref{Lemma on the representation element}. So, by $\preceqc$-transitivity, $\varphi \simeqc \psi$. So $\varphi \simeqc \varphi \wedge \psi$ as we have shown that $\preceqc$-coupling holds for all singletons. Moreover, $A \owedge B \preceqc \varphi \wedge \psi$ by $\preceqc$-counter-dominance. So $A \owedge B \preceqc \varphi \wedge \psi \simeqc \varphi \simeqc A$ and hence $A \owedge B \preceqc A$ by $\preceqc$-transitivity. Moreover, $A \preceqc A \owedge B$ by $\preceqc$-counter-dominance. Thus, $A \simeqc A \owedge B$. \\
\textit{From right to left:} Let $A \simeqc A \owedge B$ and $A \simeqc A \owedge C$. By $\preceqc$-transitivity, $A \owedge B \simeqc A \owedge C$. So $A \owedge B \simeqc (A \owedge B) \owedge (A \owedge C)$ by $\preceqc$-coupling and hence $A \simeqc (A \owedge B) \owedge (A \owedge C)$ by $\preceqc$-transitivity. In order to complete the proof, we only need to show that $(A\owedge B) \owedge (A \owedge C) \simeqc A \owedge (B \owedge C)$. 
If $A = \emptyset$ or $B = \emptyset$ or $C = \emptyset$, then $\emptyset = (A\owedge B) \owedge (A \owedge C)  \simeqc  A \owedge (B \owedge C) = \emptyset$ by $\preceqc$-counter-dominance. If they are all non-empty, it also follows immediately from $\preceqc$-counter-dominance that $(A\owedge B) \owedge (A \owedge C) \simeqc A \owedge B \owedge C$.
\end{proof}

\begin{proof}[Proof for Observation \ref{observation for reduction to single sentence}:]
\textit{1.} \textit{From left to right:} Let $A \preceqc B$. By the first item of Lemma \ref{Lemma on the representation element}, there exists some $\varphi$ in $A$ such that $\varphi \simeqc A$. Thus, $\varphi \preceqc B$ by $\preceqc$-transitivity. 
\textit{From right to left:} Suppose there exists some $\varphi \in A$ such that $\varphi \preceqc B$. By $\preceqc$-counter dominance, $A \preceqc \varphi$. Thus, $A \preceqc B$ by $\preceqc$-transitivity.\\
\textit{2.} \textit{From left to right:} Let $A \preceqc B$. By $\preceqc$-counter dominance, $B \preceqc \varphi$ for every $\varphi \in B$. Thus, by $\preceqc$-transitivity, $A \preceqc \varphi$ for all $\varphi \in B$. 
\textit{From right to left:} Let $A \preceqc \varphi$ for all $\varphi \in B$. By the first item of Lemma \ref{Lemma on the representation element}, there exists some $\psi$ in $B$ such that $\psi \simeqc B$. So $A \preceqc \psi \simeqc B$ and hence $A \preceqc B$ by $\preceqc$-transitivity.
\end{proof}

\begin{proof}[Proof for Theorem \ref{Theorem on the translation between single and multi believability relations}:]
\textit{1.} Let $\preceq$ be a quasi-linear believability relation and $\preceqc$ constructed from $\preceq$ through \btomb. We first check that $\preceqc$ is a standard multiple believability relation.
\\
\textit{$\preceqc$-determination:} It follows immediately from \btomb{ }that $\emptyset \not \preceqc A$ and $A \preceqc \emptyset$ for every non-empty $A$. Thus, $\emptyset \precc A$ for every non-empty $A$. 
\\
\textit{$\preceqc$-union:} If $A = \emptyset$ and $ B = \emptyset$, then it follows immediately from \btomb{ }that $\preceqc$-union holds for $\preceqc$. If $A \neq \emptyset$ or $ B \neq \emptyset$, then $A \cup B \neq \emptyset$. Since both $A$ and $B$ are finite and $\preceq$ satisfies $\preceq$-completeness, there exists some $\varphi \in A \cup B$ such that $\varphi \preceq \lambda$ for every $\lambda \in A \cup B$. It follows that there exists some $\varphi \in A $ such that $\varphi \preceq \lambda$ for every $\lambda \in A \cup B$ or there exists some $\varphi \in B $ such that $\varphi \preceq \lambda$ for every $\lambda \in A \cup B$. So, by \btomb, $A \preceqc A \cup B$ or $B \preceqc A \cup B$.
\\
\textit{$\preceqc$-transitivity:} Let $A \preceqc B$ and $B \preceqc C$. If $C = \emptyset$, then it follows immediately from \btomb{ }that $A \preceqc C$. If $C \neq \emptyset$,  then $A \neq \emptyset$ and $B \neq \emptyset$ since $A \preceqc B$, $B \preceqc C$ and $\preceqc$ satisfies $\preceqc$-determination as we have shown. So, by \btomb, there exists some $\varphi \in A$ such that $\varphi \preceq \psi$ for every $\psi \in B$ and there exists some $\psi \in B$ such that $\psi \preceq \lambda$ for every $\lambda \in C$. So, by $\preceq$-transitivity, there exists  some $\varphi \in A$ such that $\varphi \preceq \lambda$ for every $\lambda \in C$. Thus, $A \preceqc C$ by \btomb.
\\
\textit{$\preceqc$-counter-dominance:} Assume $(\star)$ that for every $\varphi \in B$, there exists some $\psi \in A$ such that $\varphi \vdash \psi$. If $B = \emptyset$, then it follows directly from \btomb{ }that $A \preceqc B$. If $B \neq \emptyset$,  since $B$ is finite and $\preceq$ satisfies $\preceq$-completeness, there exists some $\varphi \in B$ such that $\varphi \preceq \lambda$ for every $\lambda \in B$. Moreover, by $(\star)$, there exists some $\psi \in A$ such that $\varphi \vdash \psi$, and hence $\psi \preceq \varphi$ by $\preceq$-counter dominance. So, by $\preceq$-transitivity, $\psi \preceq \lambda$ for every $\lambda \in B$. Thus, $A \preceqc B$ by \btomb. 
\\
\textit{$\preceqc$-coupling:} Let $A \simeqc B$. If $A = \emptyset$, then it follows directly from \btomb{ }that $A \simeqc A \owedge B$. If $B = \emptyset$, then it follows immediately from $A \simeqc B = \emptyset$ that $A \simeqc A \owedge B = \emptyset$. If $A \neq \emptyset$ and $B\neq \emptyset$, since $\preceqc$ satisfies $\preceqc$-union, $\preceqc$-transitivity and $\preceqc$-counter dominance as we have shown, by the first item of  Lemma \ref{Lemma on the representation element}, there exist $\varphi \in A$ and $\psi \in B$ such that $\varphi \simeqc A$ and $\psi \simeqc B$. So, by $\preceqc$-transitivity, $\varphi \simeqc \psi$ and hence $\varphi \simeq \psi$ by \btomb. It follows that $\varphi \simeq \varphi \wedge \psi$ by $\preceq$-coupling. So $\varphi \simeqc \varphi \wedge \psi$ by \btomb. Moreover, $A \owedge B \preceqc \varphi \wedge \psi$ by $\preceqc$-counter dominance. So, $A \owedge B \preceqc \varphi \wedge \psi \simeqc \varphi \simeqc A$ and hence $A \owedge B \preceqc A$ by $\preceqc$-transitivity. It is immediate from $\preceqc$-counter dominance that $A \preceqc A \owedge B$. Thus, $A \simeqc A \owedge B$. 
\\
\textit{$\preceqc$-minimality:} \textit{From left to right:} Suppose $A \preceqc B$ for all $B$. It follows that $A \preceqc \taut$. So there exists some $\varphi \in A$ such that $\varphi \preceq \taut$ by \btomb. So $\varphi \preceq \psi$ for all $\psi$ by $\preceq$-counter dominance and $\preceq$-transitivity. So, by $\preceq$-minimality, $\varphi \in K$. Thus, $A \cap K \neq \emptyset$. \textit{From right to left:} Let $A \cap K \neq \emptyset$, i.e. there exists some $\varphi \in A \cap K$. Since $\varphi \in K$, $\varphi \preceq \psi$ for all $\psi$ by $\preceq$-minimality. So $A \preceqc B$ for all $B$ by \btomb.
\\
\textit{$\preceqc$-maximality:} Let $A$ be non-empty and $B \preceqc A$ for all non-empty $B$. Then, $\perp \preceqc A$ and hence by \btomb{ }$\falsum \preceq \varphi$ for every $\varphi \in A$. So, due to $\preceq$-counter dominance and $\preceq$-transitivity, for every $\varphi \in A$, $\psi \preceq \varphi$ for all $\varphi \in \mathcal{L}$. Hence, for every $\varphi \in A$, $\varphi \dashv \Vdash \falsum$ due to $\preceq$-maximality. Thus, $A \equiv \{\falsum\}$.\\
Let $\preceq^{\prime}$ be the believability relation derived from $\preceqc$ through \mbtob. It is easy to see that $\varphi \preceq \psi$ if and only if $\varphi \preceqc \psi$ if and only if $\varphi \preceq^{\prime} \psi$. Thus, $\preceq$ can be retrieved from $\preceqc$ through \mbtob.\\
\\
\textit{2.} It is easy to see that $\preceq$ satisfies $\preceq$-transitivity, $\preceq$-counter dominance, $\preceq$-coupling and $\preceq$-completeness. In what follows we only check $\preceq$-minimality and $\preceq$-maximality.
\\
\textit{Minimality:} \textit{From left to right:} Let $\varphi \in K$. Then, $\varphi \preceqc \psi$ for all $\psi$ due to $\preceqc$-minimality. Thus, $\varphi \preceq \psi$ for all $\psi$ by \mbtob. \textit{From right to left:} Let $\varphi \preceq \psi$ for all $\psi$. Then $\varphi \preceqc \psi$ for all $\psi$ by \mbtob. So, by the second item of Observation \ref{observation for reduction to single sentence}, $\varphi \preceqc B$ for all non-empty $B$. Moreover, due to $\preceqc$-counter-dominance, $\varphi \preceqc \emptyset$. Thus, by $\preceqc$-minimality, $\varphi \in K$.
\\
\textit{Maximality:}  Let $\psi \preceq \varphi$ for all $\psi$. Then $\psi \preceqc \varphi$ for all $\psi$ by \mbtob. So, by the first item of Observation \ref{observation for reduction to single sentence}, $B \preceqc \varphi$ for all non-empty $B$. So, due to $\preceqc$- maximality, $\varphi \dashv \Vdash \falsum $. 
\\
In order to prove that $\preceqc$ can be retrieved from $\preceq$ through \btomb, let $\preceqc^{\prime}$ be the multi-believability relation derived from $\preceq$ through \btomb. We need to show that $A \preceqc B $ if and only if $A \preceqc^{\prime} B$.
\\ 
\textit{From left to right:} Let $ A \preceqc B$. If $B = \emptyset$, it follows directly from \btomb{ }that $A \preceqc^{\prime} B$. If $A = \emptyset$, then $B=  \emptyset$ by $\preceqc$-determination. It also follows immediately from \btomb{ }that $A \preceqc^{\prime} B$. If $A \neq \emptyset$ and $B \neq \emptyset$, then, by Observation \ref{observation for reduction to single sentence}, there exists some $\varphi \in A$ such that $\varphi \preceqc \psi$ for all $\psi \in B$. So, by \mbtob, there exists some $\varphi \in A$ such that $\varphi \preceq \psi$ for all $\psi \in B$. Hence, by \btomb, $A \preceqc^{\prime} B$.
\\  
\textit{From right to left:} Let $A \preceqc^{\prime} B$. If $B = \emptyset$, then it follows directly from $\preceqc$-counter dominance that $A \preceqc B$. If $B \neq \emptyset$, then there exists some $\varphi \in A$ such that $\varphi \preceq \psi$ for all $\psi \in B$ by \btomb. So there exists some $\varphi \in A$ such that $\varphi \preceqc \psi$ for all $\psi \in B$ by \mbtob. Hence, by Observation \ref{observation for reduction to single sentence}, $A \preceqc B$.
\end{proof}

\begin{proof}[Proof for Theorem \ref{Representation theorem for choice revision based on weak multiple believability relation}:]
\textit{From construction to postulates:} $\cro$-closure: We need to prove that $K \cro A$ is a belief set for every $A$. If $\emptyset \preceqc A$, it follows from \mbtoc{ }that $K \cro A = K$. So $K \cro A$ is  belief set since $K$ is  belief set. If $\emptyset \not \preceqc A$, then $A \precc \emptyset$ by $\preceqc$-counter dominance and hence $K \cro A =\{\varphi \mid A \simeqc A \owedge \varphi \}$ by \mbtoc. Moreover, it follows from $\emptyset \not \preceqc A$ and $\preceqc$-counter dominance that $A \neq \emptyset$. So, by Lemma \ref{Lemma on the representation element}, $K \cro A =\{\varphi \mid A \simeqc A \owedge \varphi \} \neq \emptyset$. Let $\varphi \in K \cro A$, $\psi \in K \cro A$ and $\varphi \wedge \psi \vdash \lambda$, in order to complete the proof, we need to show that $\lambda \in K \cro A$. By $\preceqc$-weak coupling, it follows from $A \simeqc A \owedge \varphi$ and $A \simeqc A \owedge \psi$ that $A \simeqc A \owedge \varphi \wedge \psi$. By $\preceqc$-counter dominance and $\varphi \wedge \psi \vdash \lambda$, $A \owedge \lambda \preceqc A \owedge \varphi \wedge \psi$. So $A \owedge \lambda \preceqc A$ by $\preceqc$-transitivity. Moreover, $A \preceqc A \owedge \lambda$ by $\preceqc$-counter dominance. Thus, $A \simeqc A \owedge \lambda$, i.e. $\lambda \in K \cro A$.
\\
\textit{$\cro$-relative success:} Let $K \cro A \neq K$. Then, by \mbtoc, $A \precc \emptyset$. So, by $\preceqc$-counter dominance, $A \neq \emptyset$ and hence, by Lemma \ref{Lemma on the representation element}, there exists some $\varphi \in A$ such that $A \simeqc A \owedge \varphi$. Moreover, by \mbtoc, $K \cro A = \{\varphi \mid A \simeqc A \owedge \varphi\} $ when $A \precc \emptyset$. Thus, $A \cap K \cro A \neq \emptyset$.
\\
\textit{$\preceqc$-confirmation:} Let $A \cap K \neq \emptyset$, i.e. there exists some $\psi \in A \cap K$. Suppose $A \not \precc \emptyset$, then it follows immediately from \mbtoc{ }that $K \cro A =K$. Suppose $A \prec \emptyset$, then, by \mbtoc, we only need to show that $K = \{\varphi \mid A \simeqc A \owedge \varphi\}$. \textit{From left to right inclusion direction:} Let $\lambda \in K$. By $\preceqc$-counter dominance, $A \owedge \lambda \preceqc \lambda \wedge \psi$. Since $K$ is a belief set, it follows from $\lambda \in K$ and $\psi \in K$ that $\lambda \wedge \psi \in K$. Hence, $\lambda \wedge \psi \preceqc A$ by $\preceqc$-minimality. So $A \owedge \lambda \preceqc  A$ by $\preceqc$-transitivity. Moreover, $A \preceqc A \owedge \lambda$ by $\preceqc$-counter dominance. So, $\lambda \in \{\varphi \mid A \simeqc A \owedge \varphi \}$. \textit{From right to left inclusion direction:} Let $\lambda \in \{\varphi \mid A \simeqc A \owedge \varphi \}$. By $\preceqc$-counter dominance, $\lambda \preceqc A \owedge \lambda$ and $A \preceqc \psi$. So, $\lambda \preceqc A \owedge \lambda \simeqc A \preceqc \psi$ and hence $\lambda \preceqc \psi$ by $\preceqc$-transitivity. Since $\psi \in K$, $\psi \preceqc B$ for all $B$ by $\preceqc$-minimality. So, by $\preceqc$-transitivity, $\lambda \preceqc B$ for all $B$ and hence $\lambda \in K$ by $\preceqc$-minimality. To sum up (i) and (ii), $K = \{\varphi \mid A \simeqc A \owedge \varphi\} = K \cro A$ when $A \precc \emptyset$.
\\
\textit{$\cro$-regularity:} Let $A \cap (K \cro B) \neq \emptyset$. (i) Let $K \cro B = K$, then $A \cap K = A \cap (K \cro B) \neq \emptyset$. As we have shown that $\cro$-confirmation holds of $\cro$, it follows that $K \cro A =K$. So $A \cap (K \cro A) = A \cap K = A \cap (K \cro B) \neq \emptyset$. (ii) Let $K \cro B \neq K$, then $B \precc \emptyset$ and $K \cro B = \{\varphi \mid B \simeqc B \owedge \varphi \}$ by \mbtoc. It follows from $A \cap (K \cro B) \neq \emptyset$ and $K \cro B = \{\varphi \mid B \simeqc B \owedge \varphi \}$ that there exists some $\psi \in A$ such that $B \simeqc B \owedge \psi$. By $\preceqc$-counter dominance, $A \preceqc \psi$ and $\psi \preceqc B \owedge \psi$. So, $A \preceqc \psi \preceqc B \owedge \psi \simeqc B$ and hence $A \preceqc B$ by $\preceqc$-transitivity. Moreover, it follows from $B \precc \emptyset$ that $\emptyset \not \preceqc B$. So, by $\preceqc$-transitivity, $\emptyset \not \preceqc A $. By $\preceqc$-counter dominance, $A \preceqc \emptyset$. So, $A \precc \emptyset$ and hence $A =\{\varphi \mid A \simeqc A \owedge \varphi\}$ by \mbtoc. Also, $A \neq \emptyset$ since $A \cap (K \cro B) \neq \emptyset$. So, by Lemma \ref{Lemma on the representation element}, $A \cap (K \cro A) = A \cap \{\varphi \mid A \simeqc A \owedge \varphi\} \neq \emptyset$. To sum up (i) and (ii), $\cro$-regularity holds of $\cro$.
\\
\textit{$\cro$-reciprocity:} Let $A \cap (K \cro B) \neq \emptyset$ and $B \cap (K \cro A) \neq \emptyset$. (i) Let $K \cro A =K$ or $K \cro B = K$, then it follows immediately from $A \cap (K \cro B) \neq \emptyset$ and $B \cap (K \cro A) \neq \emptyset$ that $K \cro A = K \cro B = K$ by $\cro$-confirmation, which has been proved. Suppose $K \cro A \neq K$ and $K \cro B \neq K$, then $K \cro A =\{\varphi \mid A \simeqc A \owedge \varphi  \}$ and $K \cro B =\{\varphi \mid B \simeqc B \owedge \varphi  \}$ by \mbtoc. Next we show that $K \cro A \subseteq K \cro B$. (The converse direction can be proved in the same way.) Let $\varphi \in K \cro A = \{\varphi \mid A \simeqc A \owedge \varphi \}$. As it is immediate from $\preceqc$-counter dominance that $B \preceqc B \owedge \varphi$, we only need to show that $B \owedge \varphi \preceqc B$. Since $B \cap (K \cro A) \neq \emptyset$, there exists some $\psi \in B$ such that $A \simeqc A \owedge \psi$. By the second item of Observation \ref{Observation on the interderivability among postulates on multiple believability relations}, $\preceqc$ satisfies $\preceqc$-weak coupling, so $A \simeqc A \owedge \varphi \wedge \psi$. Due to $\preceqc$-counter dominance, $\varphi \wedge \psi \preceqc A \owedge (\varphi \wedge \psi)$. So $\varphi \wedge \psi \preceqc A$ by $\preceqc$-transitivity. Furthermore, since $A \cap (K \cro B) \neq \emptyset $, there exists some $\lambda \in A$ such that $B \simeqc B \owedge \lambda$. By $\preceqc$-counter dominance, $A \preceqc \lambda \preceqc B \owedge \lambda$. So $\varphi \wedge \psi \preceqc B$ by $\preceqc$-transitivity. Moreover, due to $\psi \in B$, $B \owedge \varphi \preceqc \varphi \wedge \psi$ by $\preceqc$-counter dominance. So, by $\preceqc$-transitivity, $B \owedge \varphi \preceqc B $. To sum up (i) and (ii), $\cro$-reciprocity holds of $\cro$. 
\\
\\
\textit{From postulates to construction:} Let $\cro$ be a choice revision operation satisfying $(\cro 1)$ through $(\cro 5)$. Let $\preceqc$ be a relation derived from $\cro$ in the following way:\begin{enumerate}
\item[] \ctomb \,\,\,\,\,\,$A \preceqc B$ iff either (i) $B \cap (K \cro B) = \emptyset$ or (ii) $A \cap (K \cro A) \neq \emptyset$ and there exists $A_0, \cdots, A_n$ such that $K \cro A= K \cro A_0$, $ K \cro B= K \cro A_n$ and $ A_0 \cap (K \cro A_1) \neq \emptyset, \cdots, A_{n-1} \cap (K \cro A_n) \neq \emptyset$.
\end{enumerate}
Let us first prove the following proposition:

\begin{enumerate}
\item[] $(\star)$ Let $\cro$ be a operation satisfying basic postulates on choice revision and $\preceqc$ constructed from $\cro$ in the way of \ctomb, then $K \cro A = K \cro B$ whenever $A \simeqc B$.
\end{enumerate}
\textit{Proof for Proposition $(\star)$}: Let $A \simeqc B$. There are two cases. (i) $B \cap (K \cro B) = \emptyset$. Then it follows from $B \preceqc A$ and \ctomb{ }that $A \cap (K \cro A) = \emptyset$. So $K \cro A = K \cro B =K$ by $\cro$-relative success. (ii) $B \cap (K \cro B) \neq \emptyset$. Then it follows from $A \preceqc B$ and \ctomb{ }that $A \cap (K \cro A) \neq \emptyset$ and there exist $A_0, \cdots, A_n$ such that $K \cro A= K \cro A_0$, $ K \cro B= K \cro A_n$ and $ A_0 \cap (K \cro A_1) \neq \emptyset, \cdots, A_{n-1} \cap (K \cro A_n) \neq \emptyset$. Since $A \cap (K \cro A) \neq \emptyset$, it follows from $B \preceqc A$ and \ctomb{ }that there exist $B_0, \cdots, B_m$ such that $K \cro B= K \cro B_0$, $ K \cro A= K \cro B_m$ and $ B_0 \cap (K \cro B_1) \neq \emptyset, \cdots, B_{m-1} \cap (K \cro B_m) \neq \emptyset$. So, it holds that $ A \cap (K \cro A_0) \neq \emptyset, \, A_0 \cap (K \cro A_1) \neq \emptyset, \cdots, A_{n-1} \cap (K \cro B) \neq \emptyset, \, B \cap (K \cro B_0) \neq \emptyset, \, B_0 \cap (K \cro B_1) \neq \emptyset, \cdots, \mbox{ and } B_{m-1} \cap (K \cro A) \neq \emptyset$. By Observation \ref{Observation that reciprocity is equivalent to strong reciprocity}, $\cro$ satisfies $\cro$-strong reciprocity. So it follows that $K \cro A =K \cro B$. 
\\
Now we turn back to check the properties of the derived $\preceqc$.\\
\textit{$\preceqc$-transitivity:} Let $A \preceqc B$ and $B \preceqc C$. There are two cases. (i) $C \cap (K \cro C) = \emptyset$. Then, it follows immediately from \ctomb{ }that $A \preceqc C$. (ii) $C \cap (K \cro C) \neq \emptyset$. Then it follows from $B \preceqc C$ and \ctomb{ }that $B \cap (K \cro B) \neq \emptyset$ and there exist $B_0, \cdots, B_n$ such that $K \cro B= K \cro B_0$, $ K \cro C= K \cro B_n$ and $ B_0 \cap (K \cro B_1) \neq \emptyset, \cdots, B_{n-1} \cap (K \cro B_n) \neq \emptyset$. Since $B \cap (K \cro B) \neq \emptyset$, it follows from $A \preceqc B$ and \ctomb{ }that $A \cap (K \cro A) \neq \emptyset$ and there exist $A_0, \cdots, A_m$ such that $K \cro A= K \cro A_0$, $ K \cro B= K \cro A_m$ and $ A_0 \cap (K \cro A_1) \neq \emptyset, \cdots, A_{m-1} \cap (K \cro A_m) \neq \emptyset$. So, by \ctomb, $A \preceqc C$.
\\
\textit{$\preceqc$-coupling:} Let $A \simeqc B$. There are three cases. (i) $A \cap (K \cro A) = \emptyset$. Then, $(A \owedge B) \cap (K \cro (A \owedge B) )= \emptyset$. Otherwise, it follows from 
$(A \owedge B) \cap (K \cro (A \owedge B)) \neq \emptyset$ that $A \cap  (K \cro (A \owedge B)) \neq \emptyset$ by $\cro$-closure, and hence $A \cap (K \cro A) \neq \emptyset$ by $\cro$-regularity, which contradicts that $A \cap (K \cro A) = \emptyset$. So, by \ctomb, $A \simeqc A \owedge B$. (ii) $B \cap (K \cro B) =\emptyset$. Then, it follows from $B \preceqc A$ and \ctomb{ }that $A \cap (K \cro A) =\emptyset$. So, this case is reducible to (i). (iii) $A \cap (K \cro A) \neq \emptyset$ and $B \cap (K \cro B) \neq \emptyset$. By Proposition $(\star)$, it follows from $A \simeqc B$ that $K \cro A =K \cro B$. So $B \cap (K \cro A) = B \cap (K \cro B) \neq \emptyset$. It follows from $A \cap (K \cro A) \neq \emptyset$ and $B \cap (K \cro A) \neq \emptyset$ that $(A \owedge B) \cap (K \cro A) \neq \emptyset$ by $\cro$-closure. So, $(A \owedge B) \cap (K \cro (A \owedge B )) \neq \emptyset$ by $\cro$-regularity, and hence $A \cap (K \cro (A \owedge B )) \neq \emptyset$ by $\cro$-closure. So, by \ctomb, $A \simeqc A \owedge B$.
\\
\textit{$\preceqc$-counter dominance:} Let $A$ and $B$ satisfy $(\dagger)$: for every $\varphi \in B$ there exists $\psi \in A$ such that $\varphi \vdash \psi$. Suppose $B \cap (K \cro B) = \emptyset$, then it follows directly from \ctomb{ }that $A \preceqc B$. Suppose $B \cap (K \cro B) \neq \emptyset$, then $A \cap (K \cro B) \neq \emptyset$ by $\dagger$ and $\cro$-closure. Furthermore, $A \cap (K \cro A) \neq \emptyset$ by $\cro$-regularity. So, by \ctomb, $A \preceqc B$.
\\
\textit{$\preceqc$-minimality:} \textit{From left to right:} Let $A \preceqc B$ for all $B$. Then, $A \preceqc \taut$. Since $\taut \in K \cro \taut$ by $\cro$-closure, it follows from $A \preceqc \taut$ and \ctomb{ }that $A \cap (K \cro A) \neq \emptyset$ and there exist $A_0, \cdots, A_n$ such that $K \cro A= K \cro A_0$, $ K \cro \taut= K \cro A_n$ and $ A_0 \cap (K \cro A_1) \neq \emptyset, \cdots, A_{n-1} \cap (K \cro A_n) \neq \emptyset$. Moreover, by $\cro$-closure, $\taut \in K \cro A_0$. By Observation \ref{Observation that reciprocity is equivalent to strong reciprocity},  $\cro$-strong reciprocity holds of $\cro$, so $K \cro A = K \cro \taut$. As $K$ is a belief set, by $\cro$-confirmation, $K \cro \taut = K$. So, $A \cap K = A \cap (K \cro \taut) = A \cap (K \cro A) \neq \emptyset $. \textit{From right to left:} Let $A \cap K \neq \emptyset$. Then $K \cro A = K$ by $\cro$-confirmation. So $A \cap (K \cro A) = A \cap K \neq \emptyset$. Moreover, since $K$ is a belief set, by $\cro$-confirmation, $K \cro \taut = K = K \cro A$. By $\cro$-closure, $\taut \in K \cro B$ for all $B$. So, by \ctomb, $A \preceqc B$ for all $B$.
\\
\textit{$\preceqc$-union:} Suppose $(A \cup B) \cap (K \cro (A \cup B)) = \emptyset$, then $\preceqc$-union follows immediately from \ctomb. Suppose $(A \cup B) \cap (K \cro (A \cup B)) \neq \emptyset$, then $A \cap (K \cro (A \cup B)) \neq \emptyset$ or $B \cap (K \cro (A \cup B)) \neq \emptyset$. Let $A \cap (K \cro (A \cup B)) \neq \emptyset$, then $A \cap (K \cro A) \neq \emptyset$ by $\cro$-regularity. So, by \ctomb, $A \preceqc A\cup B$. Similarly, $B \preceqc A \cup B$ follows from $B \cap (K \cro (A \cup B)) \neq \emptyset$. So, it holds that $A \preceqc A\cup B$ or $B \preceqc A \cup B$.
\\
Finally, we show that $\cro$ can be retrieved from $\preceqc$ through \mbtoc. Let $\cro^{\prime}$ be the operation constructed from $\preceqc$ by \mbtoc. We will show that $\cro = \cro^{\prime}$.
\\
Suppose $A \prec \emptyset$, then $A \cap (K \cro A) \neq \emptyset$ by \ctomb{ }and $K \cro A^{\prime} = \{\varphi \mid A \simeqc A \owedge \varphi\}$ by \mbtoc. (i) Let $\varphi \in K \cro A$, then $(A \owedge \varphi) \cap (K \cro A) \neq \emptyset$ by $\cro$-closure. So $(A \owedge \varphi) \cap (K \cro (A \owedge \varphi ) ) \neq \emptyset$ by $\cro$-regularity, and hence $A \cap (K \cro (A \owedge \varphi ) ) \neq \emptyset$ by $\cro$-closure. So, by \ctomb, $A \simeqc A \owedge \varphi$, i.e. $\varphi \in K \cro^{\prime} A$. (ii) Let $\varphi \in K \cro^{\prime} A$, then $A \simeqc A \owedge \varphi $. It follows from this and $A \precc \emptyset$ that $A \owedge \varphi \precc \emptyset$ by $\preceq$-transitivity. So, by \ctomb, $ (A \owedge \varphi) \cap (K \cro (A \owedge \varphi ) ) \neq \emptyset$. It follows that $\varphi \in (K \cro (A \owedge \varphi ) )$ by $\cro$-closure. Moreover, by Proposition $(\star)$, it follows from $A \simeqc A \owedge \varphi$ that $K \cro A = K \cro (A \owedge \varphi)$. So, $\varphi \in K \cro A$. To sum up (i) and (ii), $K \cro A = K \cro^{\prime} A$.
\\
Suppose $A \not \precc \emptyset$, then $K \cro^{\prime} A = K$ by \mbtoc. Moreover, due to $\preceqc$-counter dominance, it follows from $A \not \precc \emptyset$ that $A \simeqc \emptyset$. So, by Proposition $(\star)$, $K \cro A = K \cro\emptyset $. Since $\emptyset \cap (K \cro \emptyset) = \emptyset$, $K \cro \emptyset = K$ by $\cro$-relative success. So, $K \cro A= K \cro \emptyset =K =K \cro^{\prime} A$.
\\
Thus, $K \cro A = K \cro^{\prime} A$ for all $A$.
\end{proof}

\begin{proof}[Proof for Theorem \ref{Representation theorem for choice revision based on standard multiple believability relation}:]
\textit{From construction to postulates:} As we have proved Theorem \ref{Representation theorem for choice revision based on weak multiple believability relation}, here we only check $\cro$-vacuity, $\cro$-success and $\cro$-consistency.
\\
\textit{$\cro$-vacuity:} It is immediate from Observation \ref{Observation on the inter-derivability among success, relative , vacuity and regularity }.
\\
\textit{$\cro$-success:} Let $A \neq \emptyset$. By $\preceqc$-determination, $A \precc \emptyset$. So, by \mbtoc, $K \cro A =  \{\varphi \mid A \simeqc A \owedge \varphi \}$. Hence, by Lemma \ref{Lemma on the representation element}, $A \cap (K \cro A) \neq \emptyset$.
\\
\textit{$\cro$-consistency:} Let $A \not \equiv \{\falsum\}$. There are two cases. (i) $A = \emptyset$. Then $K \cro A = K$ by $\cro$-vacuity which has been shown holding of $\cro$. That $K \cro A$ is consistent follows immediately from our assumption that $K$ is consistent. (ii) $A \neq \emptyset$. Then $A \precc \emptyset$ by $\preceqc$-determination. So, $K \cro A = \{\varphi \mid A \simeqc A \owedge \varphi \}$. Suppose towards contradiction that $K \cro A \vdash \falsum$. It follows from $K \cro A \vdash \falsum$ and $\cro$-closure that $A \simeqc A \owedge \falsum$. So, by $\preceqc$ counter dominance and $\preceqc$-transitivity, $\falsum \preceqc A$ and hence $B \preceqc A$ for all non-empty $B$. So, by $\preceqc$-maximality, it follows that $A \equiv \{\falsum\}$, which contradicts $A \not \equiv \{\falsum\}$. Thus, $\cro$-consistency holds of $\cro$.
\\
\\
\textit{2. From postulates to construction:} Let $\preceqc$ be derived from $\cro$ in the way of \ctomb. Given Theorem \ref{Representation theorem for choice revision based on weak multiple believability relation}, in order to show that $\preceqc$ is a standard multiple believability relation, we only need to check $\preceqc$-maximality and $\cro$-determination.  
\\
\textit{$\preceqc$-maximality:} Let $B \neq \emptyset$ and $A \preceqc B$ for all non-empty $A$. Then, $\{\falsum\} \preceqc B$. Moreover, by $\cro$-success, $B \cap (K \cro B) \neq \emptyset$ and $\falsum \in K \cro \{\falsum\} $. So, by \ctomb, there exist $A_0, \cdots, A_n$ such that $K \cro \{\falsum\}= K \cro A_0$, $ K \cro B= K \cro A_n$ and $ A_0 \cap (K \cro A_1) \neq \emptyset, \cdots, A_{n-1} \cap (K \cro A_n) \neq \emptyset$. By $\cro$-success and closure, $K \cro \{\falsum\} = \mathcal{L}$, so $B \cap K \cro \{\falsum\} \neq \emptyset$. It follows that $K \cro B = K \cro \{\falsum\}$ since $\cro$ satisfies $\cro$-strong reciprocity according to Observation \ref{Observation that reciprocity is equivalent to strong reciprocity}. So, $K \cro B \vdash \falsum$. Hence, by $\cro$-consistency, $B \equiv \{\falsum\}$.
\\
\textit{$\preceqc$-determination:} Let $A \neq \emptyset$. It follows immediately from $\cro$-success and \ctomb{ }that $\emptyset \not \preceqc A$. Moreover, $A \preceqc \emptyset$ by $\preceqc$-counter dominance. So $A \precc \emptyset$.
\\
Furthermore, by the same argument that was presented in the proof of Theorem \ref{Representation theorem for choice revision based on weak multiple believability relation}, it follows that $\cro$ can be retrieved from $\preceqc$ through \mbtoc.
\end{proof} 
\bibliography{kappa}

\begin{thebibliography}{10}

\bibitem{alchourron_logic_1985}
Carlos~E. Alchourr{\'o}n, Peter G{\"a}rdenfors, and David Makinson.
\newblock On the logic of theory change: Partial meet contraction and revision
  functions.
\newblock {\em The Journal of Symbolic Logic}, 50(2):510--530, June 1985.

\bibitem{alchourron1982logic}
Carlos~E. Alchourr{\'o}n and David Makinson.
\newblock On the logic of theory change: Contraction functions and their
  associated revision functions.
\newblock {\em Theoria}, 48(1):14--37, 1982.

\bibitem{alchourron_logic_1985_safe}
Carlos~E. Alchourr{\'o}n and David Makinson.
\newblock On the logic of theory change: Safe contraction.
\newblock {\em Studia Logica}, 44(4):405--422, December 1985.

\bibitem{falappa_prioritized_2012}
Marcelo~A. Falappa, Gabriele Kern-Isberner, Maur{\'i}cio D.~L. Reis, and
  Guillermo~R. Simari.
\newblock Prioritized and non-prioritized multiple change on belief bases.
\newblock {\em Journal of Philosophical Logic}, 41(1):77--113, February 2012.

\bibitem{ferme_selective_1999}
Eduardo~L. Ferm{\'e} and Sven~Ove Hansson.
\newblock Selective revision.
\newblock {\em Studia Logica}, 63(3):331--342, November 1999.

\bibitem{ferme_agm_2011}
Eduardo~L. Ferm{\'e} and Sven~Ove Hansson.
\newblock {AGM} 25 years.
\newblock {\em Journal of Philosophical Logic}, 40(2):295--331, April 2011.

\bibitem{Fuhrmann_phd}
Andr{\'e} Fuhrmann.
\newblock {\em Relevant logics, modal logics and theory change}.
\newblock PhD thesis, Australian National University, September 1988.

\bibitem{gardenfors_knowledge_1988}
Peter G{\"a}rdenfors.
\newblock {\em Knowledge in Flux: {Modeling} the Dynamics of Epistemic States.}
\newblock The MIT Press, Cambridge, MA, USA, 1988.

\bibitem{gardenfors_revisions_1988}
Peter G{\"a}rdenfors and David Makinson.
\newblock Revisions of knowledge systems using epistemic entrenchment.
\newblock In {\em Proceedings of the 2nd {Conference} on {Theoretical}
  {Aspects} of {Reasoning} {About} {Knowledge}}, {TARK} '88, pages 83--95, San
  Francisco, CA, USA, 1988. Morgan Kaufmann Publishers Inc.

\bibitem{grove_two_1988}
Adam Grove.
\newblock Two modellings for theory change.
\newblock {\em Journal of Philosophical Logic}, 17(2):157--170, May 1988.

\bibitem{hansson_kernel_1994}
Sven~Ove Hansson.
\newblock Kernel contraction.
\newblock {\em Journal of Symbolic Logic}, 59(3):845--859, September 1994.

\bibitem{hansson_descriptor_2013}
Sven~Ove Hansson.
\newblock Descriptor revision.
\newblock {\em Studia Logica}, 102(5):955--980, September 2013.

\bibitem{hansson_relations_2014}
Sven~Ove Hansson.
\newblock Relations of epistemic proximity for belief change.
\newblock {\em Artificial Intelligence}, 217:76--91, December 2014.

\bibitem{hansson_descriptor_2017}
Sven~Ove Hansson.
\newblock {\em Descriptor {Revision}}, volume~46 of {\em Trends in {Logic}}.
\newblock Springer International Publishing, Cham, Switzerland, 2017.

\bibitem{jech_axiom_2008}
T.J. Jech.
\newblock {\em The {Axiom} of {Choice}}.
\newblock Dover {Books} on {Mathematics} {Series}. Dover Publications, 2008.

\bibitem{kraus_nonmonotonic_1990}
Sarit Kraus, Daniel Lehmann, and Menachem Magidor.
\newblock Nonmonotonic reasoning, preferential models and cumulative logics.
\newblock {\em Artificial Intelligence}, 44(1-2):167--207, July 1990.

\bibitem{makinson_relations_1991}
David Makinson and Peter G{\"a}rdenfors.
\newblock Relations between the logic of theory change and nonmonotonic logic.
\newblock In Andr{\'e} Fuhrmann and Michael Morreau, editors, {\em The {Logic}
  of {Theory} {Change}}, number 465 in Lecture {Notes} in {Computer} {Science},
  pages 183--205. Springer Berlin Heidelberg, 1991.

\bibitem{zhang_believability_2017}
Li~Zhang.
\newblock Believability relations for select-direct sentential revision.
\newblock {\em Studia Logica}, 105(1):37--63, February 2017.

\end{thebibliography}
\bibliographystyle{plain}

\end{document}